\documentclass[10pt]{amsart}

\usepackage[margin=1.5in]{geometry} 
\usepackage{amsmath,amsthm,amssymb,bbm,bm}
\usepackage{graphicx}
\usepackage{hyperref}
\usepackage{amssymb,amsfonts}
\usepackage{enumerate}
\usepackage{enumitem}
\usepackage{mathrsfs}
\usepackage{xcolor}
\usepackage{tikz-cd}

\newcommand{\N}{\mathbb{N}}

\newcommand{\R}{\mathbb{R}}

\newcommand{\E}{\mathbb{E}}

\newcommand{\Db}{\mathbb{D}}

\newcommand{\Var}{\mathrm{Var}}
\newcommand{\cov}{\mathrm{cov}}
\newcommand{\Beta}{\mathtt{Beta}}

\newcommand{\Gam}{\mathtt{Gamma}}
\newcommand{\Dir}{\mathtt{Dir}}
\newcommand{\PD}{\mathtt{PD}}

\newcommand{\Address}{{
\bigskip
\footnotesize

\textsc{Department of Mathematics \& Statistics, McMaster University, Hamilton, ON, L8S 4K1, Canada}\par\nopagebreak
\textit{E-mail address}: \texttt{\{shuifeng, paguyoj\}@mcmaster.ca}
}}

\def\bal#1\eal{\begin{align*}#1\end{align*}}

\newtheorem{theorem}{Theorem}[section]
\newtheorem{lemma}[theorem]{Lemma}
\newtheorem{proposition}[theorem]{Proposition}

\title[Central limit theorems and the hierarchical Dirichlet process]{Central limit theorems associated with the hierarchical Dirichlet process \hspace{-0.35cm} \normalfont{\textsuperscript{1}}}
\author{Shui Feng}
\author{J. E. Paguyo}
\date{}


\begin{document}

\begin{abstract}
The hierarchical Dirichlet process is a discrete random measure used as a prior in Bayesian nonparametrics and motivated by the study of groups of clustered data. We study the asymptotic behavior of the power sum symmetric polynomials for the vector of weights of the hierarchical Dirichlet process as the concentration parameters tend to infinity. We establish central limit theorems and obtain explicit representations for the asymptotic variances, with the latter clearly showing the impact of the hierarchical structure. These objects are related to the homozygosity in population genetics, the Simpson diversity index in ecology, and the Herfindahl-Hirschman index in economics.
\end{abstract}

\footnotetext[1]{Supported by the Natural Sciences and Engineering Research Council of Canada.}

\maketitle

\section{Introduction}

Let
\bal
\Delta  = \left\{\bm{x}= (x_1, x_2, \ldots): 0\leq x_i\leq 1, \text{for $ i=1,2,\ldots$, and } \sum_{i=1}^\infty x_i =1\right\}
\eal
be the infinite-dimensional {\em simplex}. For any $m \geq 2$ and $\bm{x} \in \Delta$, define the {\em power sum symmetric polynomials} as
\bal
\varphi_m(\bm{x})=\sum_{i=1}^\infty x_i^m. 
\eal
Consider a population of individuals of various types, with type distribution given by $\bm{x} \in \Delta$, and take a sample of size $m \geq 2$ from the population. Then $\varphi_m(\bm{x})$ is the probability that all individuals in the sample are of the same type. This observation makes $\varphi_m(\bm{x})$ a useful quantity in many contexts.  

In population genetics, individuals are categorized by their allellic types and $\varphi_m(\bm{x})$ is called the {\em homozygosity} for $m=2$ and the {\em $m$th order homozygosity} for $m>2$. The function $\varphi_2(\bm{x})$ is called the {\em Simpson index} in ecology, which is used to measure the abundance of species and the concentration of the population \cite{Simpson49}. The square root of  $\varphi_2(\bm{x})$ is called the {\em Herfindahl-Hirschman index} in economics, and is used to measure the relative size of firms and the strength of competition between them \cite{Her50, Hir45}. In the Bayesian framework, the deterministic $\bm{x}$ is replaced by random vectors, and the resulting symmetric polynomials become random indexes. The focus of this paper will be on several such models.

For any $\alpha > 0$, let $U_1, U_2, \ldots$ be independent and identically distributed $\Beta(1,\alpha)$ random variables. Let
\bal
V_1 = U_1, \qquad V_n = \prod_{k=1}^{n-1} (1 - U_k)U_n, \quad \text{for $n \geq 2$},
\eal
and let $\bm{V} = \bm{V}_\alpha = (V_1, V_2, \ldots)$. The distribution of $\bm{V}$ is called the {\em GEM distribution}, and the distribution of $\bm{V}$ reordered in descending order, $(P_1(\alpha), P_2(\alpha), \ldots) = (V_{(1)}, V_{(2)}, \ldots)$, is called the {\em Poisson-Dirichlet distribution} with concentration parameter $\alpha$, denoted by $\PD(\alpha)$. The {\em Poisson-Dirichlet distribution} was introduced by Kingman \cite{Kin75}, and it arises, among other places, as the stationary distribution of the ranked allele frequencies in the infinitely many alleles model \cite{Wat76}, as the limiting distribution of the longest cycles in a uniformly random permutation \cite{SL66}, and as the limiting distribution of the largest prime divisors of a uniformly random integer \cite{Bil72}. 

Independently of $\bm{V}$, let $S$ be a Polish space and let $M_1(S)$ denote the space of probabilities on $S$ equipped with the weak topology. For any $\nu$ in $M_1(S)$, let $\xi_1, \xi_2, \ldots$ be independent and identically distributed random variables with common distribution $\nu$. The {\em Dirichlet process} with concentration parameter $\alpha$ and base distribution $\nu$ is the random probability measure
\bal
\Xi_{\alpha, \nu} = \sum_{i=1}^\infty V_i \delta_{\xi_i}. 
\eal
The existence of the Dirichlet process was established by Ferguson \cite{Fer73}. Pitman and Yor \cite{PY97} introduced two parameter generalizations of the Poisson-Dirichlet distribution and the Dirichlet process, called the  {\em two-parameter Poisson-Dirichlet distribution} and the {\em Pitman-Yor process}, respectively. The asymptotic behavior of the Poisson-Dirichlet distribution has been intensively studied in the past few decades; see \cite{Fen10, GV17} and the references therein. 

The Dirichlet process is a fundamental prior in Bayesian nonparametrics for the study of clustered data. It is also a building block for other priors in the study of more complex data. One particular example is the hierarchical Dirichlet process, introduced by Teh et al. \cite{TJBB06}, which serves as a prior for the study of clustered groups of data. 
More specifically, for $\beta > 0$ and $n\geq 1$, let $W_n \sim \Beta(\beta V_n, \beta(1 - \sum_{k=1}^n V_k))$ and observe that $W_1, W_2, \ldots$ are conditionally independent given $\bm{V}$. Independently, let $\xi_1, \xi_2, \ldots $ be independent and identically distributed with common diffuse distribution $\nu$. Define
\bal
Z_1 = W_1, \qquad Z_n = \prod_{k=1}^{n-1} (1 - W_k)W_n, \quad \text{for $n \geq 2$},
\eal
and let $\bm{Z} = \bm{Z}_{\alpha, \beta} = (Z_1, Z_2, \ldots)$. The {\em hierarchical Dirichlet process} (henceforth HDP) restricted to one group with level two concentration parameter $\beta$, level one concentration parameter $\alpha$, and base distribution $\nu$ is the random probability measure 
\bal
\Xi_{\alpha,\beta, \nu} = \sum_{i=1}^\infty Z_i\delta_{\xi_i}. 
\eal
Simply put, the HDP is a Dirichlet process with base distribution being a draw from another Dirichlet process, or equivalently
\begin{equation}\label{HDP}
\Xi_{\alpha,\beta, \nu} \stackrel{D}{=}\Xi_{\beta, \Xi_{\alpha, \nu}},
\end{equation}
where $\stackrel{D}{=}$ denotes equality in distribution. 

Given an integer $L\geq 1$, the general HDP with $L$  groups is a collection of random measures $\Xi^1_{\alpha,\beta, \nu},\ldots, \Xi^L_{\alpha,\beta, \nu} $ that share the same Dirichlet process base measure $\Xi_{\alpha,\nu}$ and are conditionally independent and identically distributed given the base measure. We refer to the base Dirichlet process $\Xi_{\alpha,\nu}$ as the {\em level one Dirichlet process}, and $\Xi^k_{\alpha,\beta,\nu}$ for $1\leq k \leq L$ as the {\em level two Dirichlet process}. The level one Dirichlet process describes the global impact, while the level two Dirichlet processes are group specific. For more on the hierarchical Dirichlet process and other hierarchical Bayesian nonparametric models, see \cite{CLP18, Teh06, TJ10, TJBB06}. 

The hierarchical structure of the HDP makes the analysis of its asymptotic behavior more challenging, but there has been some recent progress. Camerlenghi et al. \cite{CLOP19} developed a distribution theory for hierarchical random measures generated via normalization, which encompasses the HDP and the {\em hierarchical Pitman-Yor process}. More recently, Feng \cite{Fen23} established the law of large numbers and large deviation principles for the HDP and its mass as both concentration parameters tend to infinity. 

For any $n\geq 1$, let $(W_{n1}, \ldots, W_{nn})$ be Dirichlet distributed as $\Dir(\frac{\alpha}{n}, \ldots, \frac{\alpha}{n})$, independently of $\xi_1, \xi_2, \ldots$. Let
\bal
\Xi_{\alpha,\nu;n} =\sum_{i=1}^n W_{ni}\delta_{\xi_i}.
\eal
It is known that $\Xi^n_{\alpha,\nu}$ converges in distribution to $\Xi_{\alpha,\nu}$ as $n$ tends to infinity \cite{IshZar02}. If we replace $\Xi_{\alpha,\nu}$ with $\Xi_{\alpha,\nu;n}$ in equation (\ref{HDP}), the resulting model becomes the {\em finite dimensional hierarchical Dirichlet process} (henceforth FDHDP) with dimension $n$, concentration parameter $\alpha$, and base distribution $\nu$. Denote the FDHDP as $\Xi_{\alpha,\beta, \nu;n}$. Then
\bal
\Xi_{\alpha,\beta, \nu;n} = \sum_{i=1}^n Z_{ni}\delta_{\xi_i},
\eal
where $\bm{Z}_n = (Z_{n1}, \ldots, Z_{nn})$ follows the $\Dir(\beta W_{n1}, \ldots, \beta W_{nn})$ distribution. It was shown in \cite{TJBB06} that $\Xi_{\alpha,\beta, \nu;n}$ converges in distribution to $\Xi_{\alpha,\beta, \nu}$ as $n$ tends to infinity.   

It is clear that the weight vectors $\bm{V}$, $\bm{Z}$, and $\bm{Z}_n$ are all $\Delta$-valued random vectors. For all $1\leq k\leq L$ and $n\geq 1$, let ${\bf Z}_k=(Z_{k, 1}, Z_{k,2}, \ldots)$ be such that
\bal
\Xi^k_{\alpha,\beta,\nu}=\sum_{i=1}^\infty Z_{k,i}\delta_{\xi_i},
\eal
and define the set
\bal
M_{n, L} := \{\bm{n} = (n_1,\ldots, n_L) : \text{$n_k \geq 0$ for all $1 \leq k \leq L$ and $n_1 + \dotsb + n_L = n$}\}.
\eal
The {\em $m$th order homozygosity} of the Dirichlet process, the HDP (with one group), the HDP (with $L$ groups), and the FDHDP are defined, respectively, as
\bal
H_m(\alpha) &=\varphi_m(\bm{V}) = \sum_{i=1}^\infty V_i^m, \\
H_m(\alpha,\beta) &=\varphi_m(\bm{Z}) = \sum_{i=1}^\infty Z_i^m, \\
H^L_m(\alpha,\beta) &= \sum_{i=1}^\infty \sum_{\bm{m} \in M_{m,L}} \frac{1}{L^m} \prod_{k=1}^L Z_{k,i}^{m_k}, \\
H_{m, n}(\alpha,\beta) &=\varphi_m(\bm{Z}_n) = \sum_{i=1}^n Z_{ni}^m. 
\eal
For simplicity, we will use the generic term homozygosity for all orders in the sequel. We are not aware of the study of the homozygosity of the HDP with more than one group. Our definition for the homozygosity of the HDP with $L$ groups is for random samples of size $m$ where each sample is drawn equally likely from any of the $L$ groups.  

In this paper, we consider the asymptotic behavior of the homozygosity as $\alpha$ and $\beta$ both tend to infinity. In population genetics, the homozygosity is a statistic which is used to test for departures from neutrality. Large values of $\alpha$ and $\beta$ correspond to a locus consisting of a large number of sites, with a fixed mutation rate at each site \cite{Gri79}.  In the context of Bayesian inference, the large $\alpha$ and $\beta$ limits are associated with the Dirichlet posterior limits for large sample size. In ecology, this large $\alpha$ and $\beta$ limiting regime corresponds to a system where there are a large number of species with small proportions. Our results confirm the Gaussian behavior in these regimes. 

The central limit theorem for $H_m(\alpha)$ was first established by Griffiths \cite{Gri79} for the case $m=2$ and by Joyce, Krone, and Kurtz in \cite{JKK02} for the case $m \geq 2$, where they also obtained a Gaussian limit theorem associated with the Ewens sampling formula. Subsequently they proved a central limit theorem for the homozygosity in an infinite alleles model with selective overdominance \cite{JKK03}. Handa \cite{Han09} established the central limit theorem for the homozygosity of the Pitman-Yor process, and Feng \cite{Fen10} proved a Gaussian limit theorem associated with the Pitman sampling formula. Xu \cite{Xu09, Xu11} obtained central limit theorems associated with the transformed two-parameter Poisson-Dirichlet distribution. Large deviation principles for the homozygosity were established by Dawson and Feng \cite{DF06, DF16} and moderate deviation principles by Feng and Gao \cite{FG08, FG10}. 

The objectives of this paper are to establish central limit theorems for the homozygosities $H_m(\alpha,\beta)$, $H_{m,n}(\alpha,\beta)$, and $H^L_m(\alpha,\beta)$ as the concentration parameters tend to infinity. For ease of reading, we first establish the result for the one group HDP, followed by the result for the HDP with $L$ groups, highlighting the additional complications. We also obtain law of large numbers and a large deviation principle. These results are not only theoretical generalizations of the corresponding results for the Dirichlet process, but also have strong potential for applications in many areas including ecology, economics, and population genetics.
 
\subsection{Main Results}

Let $\Gamma(x)$ be the gamma function. Let ${n \brack k}$ denote the {\em unsigned Stirling numbers of the first kind}, which are defined to be the coefficients of the {\em rising factorial}
\bal
(x)_{(n)} = \frac{\Gamma(x+n)}{\Gamma(x)} = x(x+1)\dotsb (x+n-1) = \sum_{k=0}^n {n \brack k} x^k. 
\eal
The number ${n \brack k}$ can also be defined as the number of permutations on $n$ elements with $k$ cycles. 

Our first main result is the central limit theorem for $H_m(\alpha, \beta)$. 

\begin{theorem} \label{CLTHomozygosityHDP}
Under the limiting procedure
\bal
\alpha \to \infty, \quad \beta \to \infty, \quad \frac{\alpha}{\beta} \to c \in (0,\infty),
\eal 
the homozygosity of the HDP satisfies
\bal
\tilde{H}_m(\alpha,\beta) = \sqrt{\beta} \left( \frac{H_m(\alpha, \beta) - \frac{1}{\beta^{m-1}} \sum_{j=1}^m {m \brack j} \frac{\Gamma(j)}{( \alpha/\beta )^{j-1}}}{\frac{1}{\beta^{m-1}}\sum_{j=1}^m {m \brack j} \frac{\Gamma(j)}{c^{j-1}}} \right) \xrightarrow{D} N(0, \sigma_c^2), 
\eal
where the variance is given by
\bal
\sigma_c^2 &= \frac{\sum_{j=1}^{2m} {2m \brack j} \frac{\Gamma(j)}{c^{j-1}} - \sum_{1 \leq i,j \leq m}  {m \brack i}{m \brack j} \frac{\Gamma(i+1)\Gamma(j+1)}{c^{i+j-1}}}{\left( \sum_{j=1}^m {m \brack j} \frac{\Gamma(j)}{c^{j-1}} \right)^2} - m^2.
\eal

\end{theorem}

\noindent{\bf Remarks:}
\begin{enumerate}[label=(\alph*)]
\item  The level one and level two variances are given by 
 \bal
\sigma_1^2 &= \frac{\sum_{1 \leq i,j \leq m}  {m \brack i}{m \brack j} \frac{\Gamma(i+j) - \Gamma(i+1)\Gamma(j+1)}{c^{i+j-1}}}{\left( \sum_{j=1}^m {m \brack j} \frac{\Gamma(j)}{c^{j-1}} \right)^2}, \\
\sigma_2^2 &= \frac{ \sum_{j=1}^{2m} {2m \brack j} \frac{\Gamma(j)}{c^{j-1}} - \sum_{1 \leq i,j \leq m} {m \brack i}{m \brack j} \frac{\Gamma(i+j)}{c^{i+j-1}}}{\left( \sum_{j=1}^m {m \brack j} \frac{\Gamma(j)}{c^{j-1}} \right)^2}.
\eal
The variance $\sigma_c^2$ can be written as
\bal
\sigma_c^2 &=\sigma_1^2+\sigma_2^2-m^2,
\eal
which is a mixture of individual level variations and the cross level interactions. 
\item In the case where $c$ tends to infinity, the level one concentration parameter goes to infinity faster than the level two concentration parameter, and the Gaussian fluctuation is solely determined by the level two Dirichlet process. Using the fact that ${m \brack 1} = \Gamma(m)$, we have that
\bal
\lim_{c \to \infty} \sigma_c^2 = \frac{{2m \brack 1} \Gamma(1)}{{m \brack 1}^2 \Gamma(1)^2} - m^2 = \frac{\Gamma(2m)}{\Gamma(m)^2} - m^2,
\eal
which recovers the one-level homozygosity CLT from \cite{JKK02}. 
\item If $c$ tends to zero, then $\beta$ is much bigger than $\alpha$ and the level two will fluctuate less in comparison with level one. To get the CLT, the correct scaling factor should be $\sqrt{\alpha}$, and the Gaussian behavior is solely determined by the level one Dirichlet process.
\item For any finite value of $c$, the situation is more interesting. Consider the case $m = 2$. By direct calculation, we have  $\sigma_c^2 = 2 - \frac{2(c^2 - c - 1)}{c(c+1)^2}$. For $0 < c < \frac{1+\sqrt{5}}{2}$, $\sigma^2_c$ is larger than the level one variance, which equals $2$. For $c = \frac{1+\sqrt{5}}{2}$, the two variances are the same. However for $c > \frac{1+\sqrt{5}}{2}$, the variance for the HDP limit is smaller than the level one variance.  
\end{enumerate}

Our second main result is the central limit theorem for $H_{m,n}(\alpha, \beta)$. 

\begin{theorem} \label{CLTHomozygosityFDHDP}
Under the limiting procedure
\bal
\alpha \to \infty, \quad \beta \to \infty, \quad n \to \infty, \quad \frac{\alpha}{\beta} \to c \in (0,\infty), \quad \frac{\alpha}{n} \to d \in (0,\infty), 
\eal 
the homozygosity of the FDHDP satisfies
\bal
\tilde{H}_{m,n}(\alpha,\beta)= \sqrt{\beta} \left( \frac{H_{m,n}(\alpha, \beta) - \frac{1}{\beta^{m-1}} \sum_{j=1}^m {m \brack j} \frac{\beta^{j-1} \left( \frac{\alpha}{n} + 1 \right)_{(j-1)}}{(\alpha+1)_{(j-1)}}}{\frac{1}{\beta^{m-1}}\sum_{j=1}^m {m \brack j} \frac{(d+1)_{(j-1)}}{c^{j-1}}} \right) \xrightarrow{D} N(0, v(c,d)^2), 
\eal
where the variance is given by
\bal
v(c,d)^2 &= \frac{\sum_{j=1}^{2m} {2m \brack j} \frac{(d+1)_{(j-1)}}{c^{j-1}} - \sum_{1 \leq i,j \leq m} {m \brack i}{m \brack j} \frac{(d+1)_{(i-1)}(d+1)_{(j-1)}(ij+d)}{c^{i+j-1}}}{\left( \sum_{j=1}^m {m \brack j} \frac{(d+1)_{(j-1)}}{c^{j-1}} \right)^2} - m^2.
\eal
\end{theorem}

\noindent{\bf Remarks:} 
\begin{enumerate}[label=(\alph*)] 
\item The constant $c$ is the asymptotic concentration ratio between the level one and level two Dirichlet processes, while $d$ characterizes the impact of finite dimensional approximation. 
\item Observe that $\lim_{d \to 0} v^2(c,d) = \sigma_c^2$. This means that the difference between the finite dimensional and infinite-dimensional models diminishes as $d$ tends to zero. Thus the larger the value of $d$, the stronger the finite-dimensional impact.
\end{enumerate}

In Propositions \ref{LLN}, \ref{LLNFDHDP}, and \ref{LLNDgroups}, we also obtain the law of large numbers for $H_m(\alpha, \beta)$, $H_{m,n}(\alpha, \beta)$, and $H_m^L(\alpha, \beta)$, respectively. To complete the classical trio of limit theorems, we include the following large deviation principle for $H_m(\alpha,\beta)$. 
  
\begin{theorem}\label{LDPHomozygosity}
Assume that 
\bal
\alpha \to \infty, \quad \beta \to \infty, \quad \frac{\alpha}{\beta} \to c \in (0,\infty).
\eal 
Then the  family $\{H_m(\alpha, \beta): \alpha>0, \beta>0\}$ for $m \geq 2$ satisfies a large deviation principle on $[0,1]$ with speed $\gamma = \max\{\alpha, \beta\}$ and a good rate function
\bal
J_m(q)=\inf\left\{I(\bm{x}): \bm{x} \in \Delta, \varphi_m(\bm{x})=q \right\},
\eal
where $I(\cdot)$ is given in Equation {\rm (3.1)} in \cite{Fen23}.
\end{theorem}

The proof is a direct application of the contraction principle and Theorem 3.2 in \cite{Fen23}, which gives the large deviation principle for the distribution of $\bm{Z}_{\alpha,\beta}$. 

Our final main result is the central limit theorem for $H^L_{m}(\alpha, \beta)$.

\begin{theorem} \label{CLTHomozygosityHDPgroups}
Under the limiting procedure 
\bal
\alpha \to \infty, \quad \beta \to \infty, \quad \frac{\alpha}{\beta} \to c \in (0,\infty),
\eal 
the homozygosity $H_m^L(\alpha,\beta)$ satisfies
\bal
\tilde{H}^L_{m}(\alpha, \beta) &= \frac{\sqrt{\beta}}{f(\beta; m, L, c)}\left( H^L_{m}(\alpha, \beta) - \frac{1}{L^m \beta^{m-1}}  \sum_{j = 1}^{m} A_j(m,L) \frac{\Gamma(j)}{(\alpha/\beta)^{j-1}} \right) \\
&\xrightarrow{D} N(0, \sigma_{c,L}^2)
\eal
where the variance is given by
\bal
\sigma_{c,L}^2 &= \frac{\sum_{j=1}^{2m} \tilde{A}_j(m,L) \frac{\Gamma(j)}{c^{j-1}} - \sum_{1 \leq i,j \leq m} A_i(m,L)A_j(m,L) \frac{\Gamma(i+1)\Gamma(j+1)}{c^{i+j-1}}}{\left( \sum_{j = 1}^{m} A_j(m,L) \frac{\Gamma(j)}{c^{j-1}}\right)^2} \\
&\qquad - \sum_{k=1}^L \left( \sum_{\bm{m} \in M_{m,L}} C_{\bm{m}} m_k\right)^2
\eal
with
\bal
a_j(\bm{m},L) &= \sum_{\bm{j} \in M_{j,L}} \prod_{k=1}^L {m_k \brack j_k}, \qquad A_j(m,L) = \sum_{\bm{m} \in M_{m,L}} a_j(\bm{m},L), \\
f(\beta;m, L, c)&=  \frac{1}{L^m \beta^{m-1}}  \sum_{j = 1}^{m} A_j(m,L) \frac{\Gamma(j)}{c^{j-1}},\\ 
\tilde{A}_j(m,L) &= \sum_{\bm{m}_1 \in M_{m,L}} \sum_{\bm{m}_2 \in M_{m,L}} \left( \sum_{\bm{j} \in M_{j,L}} \prod_{k=1}^L {m_{1k} + m_{2k} \brack j_k} \right) \\
C_{\bm{m}} &= \frac{\sum_{j=1}^m a_j(\bm{m},L) \frac{\Gamma(j)}{c^{j-1}}}{\sum_{j=1}^m A_j(m,L) \frac{\Gamma(j)}{c^{j-1}}}
\eal
for all $1 \leq j \leq m$, $\bm{m} \in M_{m,L}$, and $L \geq 1$. 
\end{theorem} 

\noindent {\bf Remark:}  Observe that when $L = 1$, we have that $\tilde{A}_j(m,1) = {2m \brack j}$, $A_j(m,1) = {m \brack j}$, $C_{\bm{m}} = 1$, and $m_k= m$. Thus we recover Theorem \ref{CLTHomozygosityHDP}, the central limit theorem for the single group case.

\subsection{Outline}

The outline of the paper is as follows. In Section \ref{Preliminaries}, we introduce some necessary notation and auxiliary results. In Section \ref{CLTHomozygHDP}, we prove the central limit theorem for $H_m(\alpha,\beta)$. The central limit theorem for $H_{m,n}(\alpha,\beta)$ is proved in Section \ref{CLTHomozygFDHDP} and the central limit theorem for $H^L_{m}(\alpha, \beta)$ is proved in Section \ref{CLTHomozygHDPgroups}. We conclude with some additional comments and final remarks in Section \ref{FinalRemarks}. 

\section{Preliminaries} \label{Preliminaries}

\subsection{Notation}

Let $a_n, b_n$ be two sequences. If $\lim_{n \to \infty} \frac{a_n}{b_n}$ is a nonzero constant, then we write $a_n \asymp b_n$ and we say that $a_n$ is of the same {\em order} as $b_n$. 
If the limit is $1$, we write $a_n\sim b_n$ and we say that $a_n$ is {\em asymptotic} to $b_n$.  If there exists positive constants $c$ and $n_0$ such that $a_n \leq cb_n$ for all $n \geq n_0$, then we write $a_n = O(b_n)$ and say that $a_n$ is {\em big-O} of $b_n$.

We let {\em iid} abbreviate independent and identically distributed. Convergence in distribution, in probability, and almost surely are denoted by $\xrightarrow{D}$, $\stackrel{P}{\longrightarrow}$, and $\stackrel{a.s.}{\longrightarrow}$, respectively. Similarly equality in distribution, in probability, and almost surely are denoted by $\stackrel{D}{=}$, $\stackrel{P}{=}$, and $\stackrel{a.s.}{=}$, respectively. If $X$ is a random variable distributed as $\nu$, then we write $X \sim \nu$. 

\subsection{Auxiliary Results}

Let $\Db = D([0,1])$ be the space of real-valued functions on $[0,1]$ that are right continuous and have left-hand limits. Let $B_t$ denote standard Brownian motion. Let $\nu$ be a diffuse probability measure on $[0,1]$ and set $h(t) = \nu([0,t])$. Let $\{\gamma(t) : t \geq 0\}$ be a gamma process with L\'{e}vy measure $\Lambda(dx) =  x^{-1}e^{-x}dx$. The following result will be useful. 

\begin{lemma}[Theorem 7.7, \cite{Fen10}] \label{Fen7.7}
Set
\bal
X_\theta(t) = \sqrt{\theta}\left( \frac{\gamma(\theta h(t))}{\theta} - h(t) \right), \quad X(t) = B(h(t)). 
\eal
Then $X_\theta(\cdot)$ converges in distribution to $X(\cdot)$ in $\Db$ as $\theta$ tends to infinity. 
\end{lemma}

We will also need the multivariate central limit theorem for the homozygosity associated with the Dirichlet process. 

\begin{proposition}[Theorem 1, \cite{JKK02}] \label{JKKCLT}
Suppose $\bm{X} = (X_1, X_2, \ldots) \sim \PD(\alpha)$ and let 
\bal
\tilde{H}_m(\alpha) = \sqrt{\alpha} \left( \frac{\alpha^{m-1}}{\Gamma(m)} \sum_{i=1}^\infty X_i^m - 1 \right), \qquad m \geq 2,
\eal
be the scaled $m$th order homozygosity. Then
\bal
(\tilde{H}_2(\alpha), \tilde{H}_3(\alpha), \ldots) \xrightarrow{D} (H_2, H_3, \ldots)
\eal
as $\alpha \to \infty$, where $\bm{H} = (H_2, H_3, \ldots)$ is a $\R^\infty$-valued random vector. The joint distribution of any finite number of components of $\bm{H}$ has a multivariate normal distribution with mean $\bm{0}$ and 
\bal
\cov(H_i,H_j) = \frac{\Gamma(i+j) - \Gamma(i+1)\Gamma(j+1)}{\Gamma(i)\Gamma(j)}, \qquad i,j = 2, 3, \ldots
\eal 
\end{proposition}

\section{Central Limit Theorem for the Homozygosity of the Single Group HDP} \label{CLTHomozygHDP}

In this section we prove Theorem \ref{CLTHomozygosityHDP}, the central limit theorem for $H_m(\alpha,\beta)$. This requires a series of lemmas and propositions, and we divide the presentation into several subsections. To make the development clear, we provide the following sketch of the proof. 

First, we decompose the scaled homozygosity, $\tilde{H}_m(\alpha, \beta)$, into a sum of three terms: the level 1 contribution, the level 2 contribution, and a correction term. This is done is Section \ref{SecDecomp}. In Section \ref{SecCLTLevel1}, we prove a central limit theorem for the level 1 contribution by rewriting it as a linear combination of the scaled one-level homozygosities, $\tilde{H}_j(\alpha)$, and then applying the Joyce-Krone-Kurtz multivariate central limit theorem. Next we use characteristic functions and a conditional central limit theorem to prove a central limit theorem for the level 2 contribution in Section \ref{SecCLTLevel2}. The last ingredient is a multivariate central limit theorem for the level 2 contribution and the correction term, which we give in Section \ref{SecJointConvergence}. Finally we combine the above and prove the central limit theorem for $\tilde{H}_m(\alpha,\beta)$ in Section \ref{SecProofofCLTHDP}. 

\subsection{Gamma Representation, Law of Large Numbers, and Decomposition} \label{SecDecomp}

Let $\gamma(t)$ be the gamma process above. Set
\bal
\gamma_k(\alpha,\beta)= \gamma\left(\beta \sum_{i=1}^k V_i \right)- \gamma\left(\beta\sum_{i=1}^{k-1}V_i \right),
\eal
for $k \in \N$. By the Gamma-Dirichlet algebra \cite{Pitman06}, we have 
\bal
Z_1\stackrel{D}{=}\frac{\gamma_1(\alpha, \beta)}{\gamma(\beta)}\sim \Beta(\beta V_1, \beta(1-V_1))
\eal
and 
\bal
Z_k \stackrel{D}{=} \frac{\sum_{i=2}^\infty \gamma_i(\alpha,\beta)}{\sum_{i=1}^\infty\gamma_i(\alpha,\beta)}\cdot \frac{\sum_{i=3}^\infty\gamma_i(\alpha,\beta)}{\sum_{i=2}^\infty\gamma_i(\alpha,\beta)}\cdots  \frac{\gamma_k(\alpha,\beta)}{\sum_{i=k}^\infty\gamma_i(\alpha,\beta)} \stackrel{D}{=} \frac{\gamma_k(\alpha,\beta)}{\gamma(\beta)}. 
\eal

The following lemma gives the moments of the $Z_k$ weights. 

\begin{lemma} \label{Zmoments}
For all $k$ and all $m \geq 1$, 
\bal
\mathbb{E}(Z_k^m) = \frac{1}{(\beta)_{(m)}} \sum_{j=1}^m {m \brack j} \beta^j \frac{j!}{(\alpha+1)_{(j)}} \left( \frac{\alpha}{\alpha + j} \right)^{k-1}.
\eal
\end{lemma}

\begin{proof}
The conditional expectation of $Z_k^m$ given $\bm{V}$ is
\bal
\mathbb{E}(Z_k^m \vert V_1, \ldots, V_k) = \frac{(\beta V_k)_{(m)}}{(\beta)_{(m)}} = \frac{1}{(\beta)_{(m)}} \sum_{j=1}^m {m \brack j} \beta^j V_k^j
\eal
so that
\bal
\mathbb{E}(Z_k^m) = \mathbb{E}(\mathbb{E}(Z_k^m \vert V_1, \ldots, V_k)) = \frac{1}{(\beta)_{(m)}} \sum_{j=1}^m {m \brack j} \beta^j \mathbb{E}(V_k^j). 
\eal
The expectation of $V_k^j$ is
\bal
\mathbb{E}(V_k^j) &= \mathbb{E}\left(U_k^j \prod_{i=1}^{k-1} (1 - U_i)^j\right) = \mathbb{E}(U_k^j) \prod_{i=1}^{k-1} \mathbb{E}[(1-U_k)^j] \\
&= \frac{(1)_{(j)}}{(1 + \alpha)_{(j)}} \left(\frac{(\alpha)_{(j)}}{(\alpha + 1)_{(j)}}\right)^{k-1} = \frac{j!}{(\alpha+1)_{(j)}} \left( \frac{\alpha}{\alpha + j} \right)^{k-1}.
\eal
Plugging this back into the formula for $\mathbb{E}(Z_k^m)$ finishes the proof. 
\end{proof}

We can now calculate the expectation of $H_m(\alpha, \beta)$. 

\begin{lemma}\label{homozygositymean}
For all $m\geq 2$,  
\bal
\mathbb{E}(H_m(\alpha, \beta))  = \frac{1}{(\beta)_{(m)}} \sum_{j=1}^m {m \brack j} \beta^j \frac{\Gamma(j)}{(\alpha+1)_{(j-1)}}.
\eal
\end{lemma}

\begin{proof}
Using Lemma \ref{Zmoments},
\bal
\mathbb{E}(H_m(\alpha, \beta)) &= \sum_{k=1}^\infty \mathbb{E}(Z_k^m) = \sum_{k=1}^\infty \frac{1}{(\beta)_{(m)}} \sum_{j=1}^m {m \brack j} \beta^j \frac{j!}{(\alpha+1)_{(j)}} \left( \frac{\alpha}{\alpha + j} \right)^{k-1} \\
&= \frac{1}{(\beta)_{(m)}} \sum_{j=1}^m {m \brack j} \beta^j \frac{j!}{(\alpha+1)_{(j)}} \sum_{k=0}^\infty \left( \frac{\alpha}{\alpha + j} \right)^k \\
&= \frac{1}{(\beta)_{(m)}} \sum_{j=1}^m {m \brack j} \beta^j \frac{\Gamma(j)}{(\alpha+1)_{(j-1)}}. \qedhere
\eal
\end{proof}

{\bf Remark:}  Lemma \ref{homozygositymean} is a special case of Theorem 1 in \cite{CLP18}; see their Example 1 with $k = 1$ and $n_1 = m$.

Observe that the mean of the homozygosity is asymptotically
\bal
\mathbb{E}(H_m(\alpha, \beta)) \sim \frac{1}{\beta^{m-1}} \sum_{j=1}^m {m \brack j} \Gamma(j) \left(\frac{\beta}{\alpha}\right)^{j-1} \sim \frac{1}{\beta^{m-1}} \sum_{j=1}^m {m \brack j} \frac{\Gamma(j)}{c^{j-1}}
\eal
for large $\alpha, \beta$. For notational simplicity, define 
\bal
f(\beta;m,c) = \frac{1}{\beta^{m-1}} \sum_{j=1}^m {m \brack j} \frac{1}{c^{j-1}} \Gamma(j).
\eal 
The next proposition provides a law of large numbers for the homozygosity. 

\begin{proposition}\label{LLN}
The homozygosity satisfies 
\bal
\frac{H_m(\alpha, \beta)}{f(\beta; m,c)} \to 1
\eal
in probability as $\alpha, \beta \to \infty$ such that $\frac{\alpha}{\beta} \to c$. 
\end{proposition}

\begin{proof} By Chebyshev's inequality it suffices to show that the variance of $\frac{H_m(\alpha, \beta)}{f(\beta; m,c)}$ converges to zero. Applying the gamma representation and the Gamma-Dirichlet algebra, we obtain
\bal
\mathbb{E}[H_m^2(\alpha,\beta)] &= \left(\mathbb{E}[\gamma(\beta)^{2m}]\right)^{-1}\left\{\mathbb{E}\left[\sum_{i=1}^\infty \gamma_i^{2m}(\alpha,\beta) \right] + \mathbb{E}\left[ \sum_{i,j=1}^\infty \mathbb{E}[\gamma_i^m(\alpha,\beta) \gamma_j^m(\alpha,\beta) \mid \bm{V}] \right] \right. \\
&\qquad \left. - \mathbb{E}\left[\sum_{i=1}^\infty \left(\mathbb{E}[\gamma_i^m(\alpha,\beta) \mid \bm{V}]\right)^2\right] \right\} \\
&= \mathbb{E}[H_{2m}(\alpha,\beta)]+\mathbb{E}\left[\left(\mathbb{E}[H_m(\alpha,\beta)\mid \bm{V}]\right)^2\right] -\mathbb{E}\left[\sum_{i=1}^\infty \left(\mathbb{E}[Z_i^m \mid \bm{V}]\right)^2\right].
\eal 
By Lemma \ref{homozygositymean}, we have
\bal
\mathbb{E}[H_{2m}(\alpha,\beta)]\asymp \beta^{-(2m-1)}.
\eal
Noting that
\bal
\mathbb{E}[Z_i^m \mid \bm{V}]  = \frac{1}{(\beta)_{(m)}} \sum_{j=1}^m {m \brack j} \beta^j V_i^j,
\eal
it follows that
\bal
\mathbb{E}\left[\sum_{i=1}^\infty \left(\mathbb{E}[Z_i^m \mid \bm{V}]\right)^2\right] &=\left( \frac{1}{(\beta)_{(m)}}\right)^2 \sum_{i=1}^\infty \mathbb{E}\left[\sum_{j,l=1}^m {m \brack j}{m \brack l} \beta^{j+l} V_i^{j+l}\right] \\
&= \left( \frac{1}{(\beta)_{(m)}}\right)^2 \sum_{i=1}^\infty \left[\sum_{j,l=1}^m {m \brack j}{m \brack l} \beta^{j+l} \left(\frac{\alpha}{\alpha+j+l}\right)^{i-1}\frac{(j+l)!}{(\alpha+1)_{(j+l)}}\right]\\
&\asymp \beta^{-(2m-1)}.
\eal

On the other hand,
\bal
\mathbb{E}[H_m(\alpha,\beta) \mid \bm{V}] = \frac{1}{(\beta)_{(m)}} \sum_{j=1}^m {m \brack j} \beta^j\sum_{i=1}^\infty V_i^j =  \frac{1}{(\beta)_{(m)}} \sum_{j=1}^m {m \brack j} \beta^j H_j(\alpha),
\eal
which implies that
\bal
&\mathbb{E}\left[\left(\mathbb{E}[H_m(\alpha,\beta) \mid \bm{V}]\right)^2\right] \\
&= \left(\frac{1}{(\beta)_{(m)}}\right)^2 \sum_{i, j=1}^m {m \brack i}{m \brack j} \beta^{i+j} \mathbb{E}[H_i(\alpha)H_j(\alpha)]\\
&= \left(\frac{1}{(\beta)_{(m)}}\right)^2 \sum_{i, j=1}^m {m \brack i}{m \brack j} \beta^{i+j} \sum_{l,r=1}^\infty \mathbb{E}[V_l^i V_r^j]\\
&= \frac{1}{(\beta)^2_{(m)}} \sum_{i, j=1}^m {m \brack i}{m \brack j} \beta^{i+j} \frac{\alpha \Gamma(i+1)}{(j+\alpha)_{(i+1)}}\frac{\Gamma(j+1)}{(\alpha+1)_{(j)}}\\
& \times \left[\sum_{1\leq l<r<\infty} \left(\frac{\alpha}{i+j+\alpha}\right)^{l-1}  \left(\frac{\alpha}{j+\alpha}\right)^{r-l-1}+\sum_{1\leq r<l<\infty} \left(\frac{\alpha}{i+j+\alpha}\right)^{r-1}  \left(\frac{\alpha}{i+\alpha}\right)^{l-r-1}\right]\\
&= \frac{1}{(\beta)^2_{(m)}} \sum_{i, j=1}^m {m \brack i}{m \brack j} \beta^{i+j} \frac{\alpha \Gamma(i)}{(j+\alpha)_{(i)}}\frac{\Gamma(j)}{(\alpha+1)_{(j-1)}} \\
& \sim f(\beta;m,c)^2.
\eal
Combining the above yields the result.
\end{proof}

We now give a decomposition of the scaled homozygosity $\tilde{H}_m(\alpha, \beta)$. Let
\bal
X_m(\alpha, \beta) &=\frac{\sqrt{\beta}}{f(\beta;m,c)} \sum_{k=1}^\infty \frac{\gamma_k^m(\alpha, \beta) - \frac{\Gamma(\beta V_k + m)}{\Gamma(\beta V_k)}}{\beta^m} \\
Y_m(\alpha, \beta) &= \frac{\sqrt{\beta}}{f(\beta;m,c)} \left( \sum_{k=1}^\infty \frac{\Gamma(\beta V_k + m)}{\beta^m \Gamma(\beta V_k)} - \frac{1}{\beta^{m-1}} \sum_{j=1}^m {m \brack j} \frac{\Gamma(j)}{( \alpha/\beta )^{j-1}} \right).
\eal
Then for all $m \geq 2$, 
\bal
\tilde{H}_m(\alpha,\beta) &= X_m(\alpha, \beta) + Y_m(\alpha, \beta) \\
&+ \sqrt{\beta} \left( \left( \frac{\beta}{\gamma(\beta)} \right)^m - 1 \right) \frac{\sum_{k=1}^\infty\frac{\gamma_k^m(\alpha, \beta)}{\beta^m}}{f(\beta; m,c)}.
\eal
We refer to $X_m(\alpha, \beta)$ as the {\em level 2 contribution} and $Y_m(\alpha, \beta)$ as the {\em level 1 contribution}. 

\subsection{CLT for the Level 1 Contribution} \label{SecCLTLevel1}

In this section, we prove a central limit theorem for the level 1 contribution. 

\begin{lemma} \label{Level1CLTHDP}
The level 1 contribution satisfies
\bal
Y_m(\alpha, \beta) = \frac{\sqrt{\beta}}{f(\beta;m,c)} \left( \sum_{k=1}^\infty \frac{\Gamma(\beta V_k + m)}{\beta^m \Gamma(\beta V_k)} - \frac{1}{\beta^{m-1}} \sum_{j=1}^m {m \brack j} \frac{\Gamma(j)}{( \alpha/\beta )^{j-1}} \right) \xrightarrow{D} N(0, \sigma_1^2)
\eal
as $\alpha, \beta \to \infty$ such that $\frac{\alpha}{\beta} \to c$, where the variance is given by
\bal
\sigma_1^2 = \frac{\sum_{1 \leq i,j \leq m}  {m \brack i}{m \brack j} \frac{\Gamma(i+j) - \Gamma(i+1)\Gamma(j+1)}{c^{i+j-1}}}{\left( \sum_{j=1}^m {m \brack j} \frac{\Gamma(j)}{c^{j-1}} \right)^2}.
\eal
\end{lemma}

\begin{proof}
By direct calculation we have 
\bal
Y_m(\alpha, \beta) &= \frac{\sqrt{\beta}}{\sum_{j=1}^m {m \brack j} \frac{\Gamma(j)}{c^{j-1}}} \left( \sum_{j=1}^m {m \brack j} \beta^{j-1} \sum_{k=1}^\infty V_k^j - \sum_{j=1}^m {m \brack j} \frac{\Gamma(j)}{( \alpha/\beta )^{j-1}} \right) \\
&= \sqrt{\frac{\beta}{\alpha}} \left( \frac{1}{ \sum_{j=1}^m {m \brack j} \frac{\Gamma(j)}{c^{j-1}}} \sum_{j=2}^{m} {m \brack j}\frac{\Gamma(j)}{( \alpha/\beta )^{j-1}} \left[ \sqrt{\alpha} \left( \frac{\alpha^{j-1}}{\Gamma(j)} \sum_{k=1}^\infty V_k^j - 1 \right) \right] \right) \\
&=\sqrt{\frac{\beta}{\alpha}} \left( \frac{1}{ \sum_{j=1}^m {m \brack j} \frac{\Gamma(j)}{c^{j-1}}} \sum_{j=2}^{m} {m \brack j}\frac{\Gamma(j)}{( \alpha/\beta )^{j-1}} \tilde{H}_j(\alpha) \right),
\eal
where $\tilde{H}_j(\alpha) = \sqrt{\alpha} \left( \frac{\alpha^{j-1}}{\Gamma(j)}H_j(\alpha) - 1 \right)$ is the scaled level 1 homozygosity. 

By Proposition \ref{JKKCLT},
\bal
(\tilde{H}_2(\alpha), \tilde{H}_3(\alpha),\ldots) \xrightarrow{D} (H_2,H_3,\ldots)
\eal
as $\alpha \to \infty$, where for each $m \geq 2$, $(H_2,\ldots, H_m)$ is a multivariate normal distribution with mean vector $0$ and covariance matrix
\bal
\cov(H_i, H_j)= \frac{\Gamma(i+j) - \Gamma(i+1)\Gamma(j+1)}{\Gamma(i)\Gamma(j)}
\eal
for $2 \leq i,j \leq m$. 

Since $Y_m(\alpha, \beta)$ is a linear combination of $(\tilde{H}_2(\alpha), \ldots, \tilde{H}_m(\alpha))$, which is asymptotically multivariate normal, it follows that $Y_m(\alpha, \beta)$ converges in distribution to a univariate normal with mean $0$ and variance
\bal
\sigma_1^2 &= \lim_{\alpha, \beta \to \infty} \frac{\beta/\alpha}{\left( \sum_{j=1}^m {m \brack j} \frac{\Gamma(j)}{c^{j-1}} \right)^2} \Var\left( \sum_{j=2}^m {m \brack j} \frac{\Gamma(j)}{( \alpha/\beta )^{j-1}} \tilde{H}_j(\alpha)  \right) \\
&= \lim_{\alpha, \beta \to \infty} \frac{\beta/\alpha}{\left( \sum_{j=1}^m {m \brack j} \frac{\Gamma(j)}{c^{j-1}} \right)^2} \sum_{2 \leq i,j \leq m}  {m \brack i}{m \brack j} \frac{\Gamma(i)\Gamma(j)}{(\alpha/\beta)^{i+j-2}} \cov\left( \tilde{H}_i(\alpha), \tilde{H}_j(\alpha) \right) \\
&= \frac{1}{c \left( \sum_{j=1}^m {m \brack j} \frac{\Gamma(j)}{c^{j-1}} \right)^2} \sum_{2 \leq i,j \leq m}  {m \brack i}{m \brack j} \frac{\Gamma(i)\Gamma(j)}{c^{i+j-2}} \cov\left( H_i, H_j \right) \\
&= \frac{\sum_{2 \leq i,j \leq m}  {m \brack i}{m \brack j} \frac{\Gamma(i+j) - \Gamma(i+1)\Gamma(j+1)}{c^{i+j-1}}}{\left( \sum_{j=1}^m {m \brack j} \frac{\Gamma(j)}{c^{j-1}} \right)^2}.  
\eal
Noting that $(\Gamma(2) - \Gamma(2)\Gamma(2)) + \sum_{1 < j \leq m} (\Gamma(1 + j) - \Gamma(2)\Gamma(j + 1)) + \sum_{1 < i \leq m} (\Gamma(i + 1) - \Gamma(i+1)\Gamma(2)) = 0$, we can extend the indices of the summation to get
\bal
\sigma_1^2 &= \frac{\sum_{1 \leq i,j \leq m}  {m \brack i}{m \brack j} \frac{\Gamma(i+j) - \Gamma(i+1)\Gamma(j+1)}{c^{i+j-1}}}{\left( \sum_{j=1}^m {m \brack j} \frac{\Gamma(j)}{c^{j-1}} \right)^2}. \qedhere
\eal
\end{proof}

\subsection{CLT for the Level 2 Contribution} \label{SecCLTLevel2}

In this section, we prove a central limit theorem for the level 2 contribution, $X_m(\alpha, \beta)$. We begin with a series of lemmas that we will need, starting with a result on the moments of a gamma random variable. 

\begin{lemma} \label{Gammamoments}
Let $X \sim \Gam(s,1)$ where $s > 0$ and let $m \geq 2$. Then
\bal
&\mathbb{E}[X^m] = \frac{\Gamma(s + m)}{\Gamma(s)}, \\
&\mathbb{E}[(X^m - \mathbb{E}(X^m)^2] = \mathbb{E}[X^{2m}] - (\mathbb{E}[X^m])^2 = \frac{\Gamma(s+2m)}{\Gamma(s)} - \left( \frac{\Gamma(s+m)}{\Gamma(s)} \right)^2,
\eal
and for $s$ sufficiently large, 
\bal
s^{-2m} \mathbb{E}[(X^m - \mathbb{E}[X^m])^2] &= \frac{m^2}{s} + O\left( \frac{1}{s^2} \right) \\
s^{-\ell m} \mathbb{E}[(X^m - \mathbb{E}[X^m])^\ell] &= O\left( \frac{1}{s^2} \right), \qquad \text{for $\ell \geq 3$}.
\eal
\end{lemma}

\begin{proof}
The proof follows from direct computation and the fact that for large $s$ and $\ell \geq 1$,
\bal
&\left[ \frac{\Gamma(s+m)}{\Gamma(s)} \right]^\ell = s^{\ell m}\left(1 + \frac{\ell m(m-1)}{2s} + O\left( \frac{\ell m}{12s^2} \right)\right). \qedhere
\eal 
\end{proof}

The next standard result will be useful in proving the convergence of our characteristic functions. 

\begin{lemma}[Exercise 3.1.1, \cite{Dur19}] \label{expconv}
If $\max_{1 \leq k \leq n} |c_{n,k}| \to 0$, $\sum_{k=1}^n c_{n,k} \to \lambda \in \R$, and $\sup_n \sum_{k=1}^n |c_{n,k}| < \infty,$ then $\prod_{k=1}^n (1 + c_{n,k}) \to e^\lambda$. 
\end{lemma}

\begin{proof}
Observe that
\bal
\left| \sum_{k=1}^n \log(1 + c_{n,k}) - \sum_{k=1}^n c_{n,k} \right| &\leq \sum_{k=1}^n |c_{n,k}| \left| \frac{\log(1 + c_{n,k})}{c_{n,k}} - 1 \right| \\ 
&\leq \left( \max_{1 \leq k \leq n} \left| \frac{\log(1 + c_{n,k})}{c_{n,k}} - 1 \right|  \right) \sum_{k=1}^n |c_{n,k}| \\
& \to 0
\eal
as $n \to \infty$ by the hypotheses. Therefore $\sum_{k=1}^n \log(1 + c_{n,k}) \to \lambda$ as $n \to \infty$, and the result follows. 
\end{proof}

Finally we will need the following strong law of large numbers for the homozygosity of the Dirichlet process. 

\begin{lemma}\label{as-LLN}
For all $j \geq 1$, 
\bal
\frac{\alpha^{j-1}}{\Gamma(j)}H_j(\alpha) \to 1
\eal
almost surely as $\alpha \to \infty$. Furthermore, this almost sure convergence also holds if $H_j(\alpha)$ is replaced by $\sum_{k=1}^{\lfloor\alpha^2 \rfloor} V_k^j$.
\end{lemma}

\begin{proof} By direct calculation, 
\bal
\mathbb{E}[H_j(\alpha)]&= \frac{\Gamma(j)}{(\alpha+1)_{(j-1)}}, \\
\mathbb{E}[H^2_j(\alpha)] & = \sum_{k,l=1}^\infty\mathbb{E}[V_k^j V_l^j]=\frac{\alpha \Gamma(j)^2}{(j+\alpha)_{(j)}(\alpha+1)_{(j-1)}},
\eal
and
\bal
\mathbb{E}\left[\left(\frac{\alpha^{j-1}}{\Gamma(j)}H_j(\alpha)-1\right)^2\right] &=\frac{\alpha^{(2j-1)}}{(j+\alpha)_{(j)}(\alpha+1)_{j-1}} +1 - 2 \frac{\alpha^{j-1}}{(\alpha+1)_{(j-1)}} \asymp \alpha^{-1}. 
\eal

For any $a>1$, set $\alpha_n = \lfloor a^n\rfloor$. Then by the Borel-Cantelli lemma,
\bal
\frac{\alpha_n^{j-1}}{\Gamma(j)}H_j(\alpha_n) \to 1
\eal
almost surely as $n \to \infty$.

For general $\alpha$, we can find $n$ such that $\alpha_n\leq \alpha\leq \alpha_{n+1}$. Applying the gamma representation for $\bm{V}$, we have that
\bal
\frac{\alpha_n}{\alpha}\frac{\gamma(\alpha_n)/\alpha_n}{\gamma(\alpha)/\alpha} \frac{\alpha_n^{j-1}}{\Gamma(j)}H_j(\alpha_n)\leq \frac{\alpha^{j-1}}{\Gamma(j)}H_j(\alpha)\leq \frac{\alpha_{n+1}}{\alpha}\frac{\gamma(\alpha_{n+1})/\alpha_{n+1}}{\gamma(\alpha)/\alpha} \frac{\alpha_{n+1}^{j-1}}{\Gamma(j)}H_j(\alpha_{n+1}).\eal
The almost sure convergence of $H_j(\alpha)$ follows by letting $n$ tend to infinity followed by letting $a$ converge to $1$. 

Finally the truncated version follows from the observation that 
\bal
\frac{\alpha^{j-1}}{\Gamma(j)}\sum_{k=\lfloor \alpha^2\rfloor+1}^{\infty}V_k^j \to 0
\eal
almost surely as $\alpha \to \infty$.
\end{proof}

We are now ready to prove the central limit theorem for the level 2 contribution. 

\begin{lemma} \label{Level2CLTHDP}
The level 2 contribution satisfies
\bal
X_m(\alpha, \beta) = \frac{\sqrt{\beta}}{f(\beta; m,c)} \sum_{k=1}^\infty \frac{\gamma_k^m(\alpha,\beta) - \frac{\Gamma(\beta V_k + m)}{\Gamma(\beta V_k)}}{\beta^m} \xrightarrow{D} N(0, \sigma_2^2),
\eal
as $\alpha,\beta \to \infty$ such that $\frac{\alpha}{\beta} \to c$, where the variance is given by
\bal
\sigma_2^2 = \frac{ \sum_{j=1}^{2m} {2m \brack j} \frac{\Gamma(j)}{c^{j-1}} - \sum_{1 \leq i,j \leq m} {m \brack i}{m \brack j} \frac{\Gamma(i+j)}{c^{i+j-1}}}{\left( \sum_{j=1}^m {m \brack j} \frac{\Gamma(j)}{c^{j-1}} \right)^2}.
\eal
\end{lemma}

\begin{proof} We show that the characteristic function of $X_m(\alpha, \beta)$ converges to the characteristic function of a normal random variable, $N(0, \sigma_2^2)$. Define
\bal
X_{m1}(\alpha, \beta) &= \frac{\sqrt{\beta}}{f(\beta; m,c)} \sum_{k=1}^{\lfloor \alpha^2 \rfloor}\frac{\gamma_k^m(\alpha,\beta) - \frac{\Gamma(\beta V_k + m)}{\Gamma(\beta V_k)}}{\beta^m}, \\
X_{m2}(\alpha, \beta) &= \frac{\sqrt{\beta}}{f(\beta; m,c)} \sum_{k=\lfloor \alpha^2 \rfloor+1}^\infty \frac{\gamma_k^m(\alpha,\beta) - \frac{\Gamma(\beta V_k + m)}{\Gamma(\beta V_k)}}{\beta^m},
\eal
and observe that
\bal
X_m(\alpha,\beta) = X_{m1}(\alpha,\beta) + X_{m2}(\alpha,\beta). 
\eal

First we show that $X_{m2}(\alpha,\beta)$ converges to $0$. Noting that
\bal
\mathbb{E}\left[\sum_{k=\lfloor \alpha^2 \rfloor+1}^\infty\frac{\Gamma(\beta V_k+m)}{\Gamma(\beta V_k)}\right] = \sum_{j=1}^m {m\brack j}\frac{ \Gamma(j) \beta^j}{(\alpha+1)_{(j-1)}} 
\left(\frac{\alpha}{\alpha + j}\right)^{\lfloor \alpha^2 \rfloor} 
\eal
it follows by Chebyshev's inequality that
\bal
P\left(|X_{m2}(\alpha,\beta)| > \epsilon \right) &\leq P\left(\frac{\sqrt{\beta}}{\beta^m f(\beta; m,c)} \sum_{k=\lfloor \alpha^2 \rfloor+1}^\infty \left[\gamma_k^m(\alpha,\beta) +\frac{\Gamma(\beta V_k + m)}{\Gamma(\beta V_k)} \right]> \epsilon\right)\\
&\leq \frac{2}{\epsilon \sqrt{\beta}} \left( \sum_{j=1}^m {m \brack j} \frac{\Gamma(j)}{c^{j-1}} \right)^{-1}\sum_{k=\lfloor \alpha^2 \rfloor+1}^\infty\mathbb{E}\left[\frac{\Gamma(\beta V_k + m)}{\Gamma(\beta V_k)} \right]\\
&=\frac{2}{\epsilon \sqrt{\beta}} \left( \sum_{j=1}^m {m \brack j} \frac{\Gamma(j)}{c^{j-1}} \right)^{-1}  \sum_{j=1}^m {m\brack j}\frac{ \Gamma(j) \beta^j }{(\alpha+1)_{(j-1)}} 
\left(\frac{\alpha}{\alpha+j}\right)^{\lfloor \alpha^2 \rfloor} \\
&\to 0
\eal
as $\alpha, \beta \to \infty$ and $\alpha/\beta \to c$. 

Next, conditioning on $\bm{V}$, we can write the characteristic function of $X_{m1}(\alpha, \beta)$ as
\bal
&\mathbb{E}\left( \exp\left( it \frac{\sqrt{\beta}}{f(\beta; m,c)} \sum_{k=1}^{\lfloor \alpha^2 \rfloor} \frac{\gamma^m_k(\alpha, \beta) - \frac{\Gamma(\beta V_k + m)}{\Gamma(\beta V_k)}}{\beta^m} \right) \right) \\
&\qquad = \mathbb{E}\left[ \mathbb{E}\left( \exp\left( it \frac{\sqrt{\beta}}{f(\beta; m,c)} \sum_{k=1}^{\lfloor \alpha^2\rfloor} \frac{\gamma^m_k(\alpha, \beta) - \frac{\Gamma(\beta V_k + m)}{\Gamma(\beta V_k)}}{\beta^m} \right) \middle\vert \bm{V} \right) \right] \\
&\qquad = \mathbb{E}\left[ \prod_{k=1}^{\lfloor \alpha^2 \rfloor} \mathbb{E}\left( \exp\left( it \frac{\sqrt{\beta}}{f(\beta; m,c)} \frac{\gamma^m_k(\alpha,\beta) - \frac{\Gamma(\beta V_k + m)}{\Gamma(\beta V_k)}}{\beta^m}  \right) \middle\vert \bm{V} \right) \right].
\eal
Using Lemma \ref{Gammamoments}, the inner conditional characteristic function can be computed as
\bal
&\mathbb{E}\left( \exp\left( it \frac{\sqrt{\beta}}{f(\beta; m,c)} \frac{\gamma^m_k(\alpha, \beta) - \frac{\Gamma(\beta V_k + m)}{\Gamma(\beta V_k)}}{\beta^m}  \right) \middle\vert \bm{V} \right) \\
&= \mathbb{E}\left( \sum_{\ell = 0}^\infty \frac{ \left( it \frac{\sqrt{\beta}}{f(\beta; m,c)} \right)^\ell}{\ell!} \beta^{-\ell m} \left[ \gamma^m_k(\alpha, \beta) - \frac{\Gamma(\beta V_k + m)}{\Gamma(\beta V_k)} \right]^\ell \middle\vert \bm{V}  \right)  \\
&=  \sum_{\ell = 0}^\infty \frac{ \left( it \frac{\sqrt{\beta}}{f(\beta; m,c)} \right)^\ell}{\ell!} V_k^{\ell m} \mathbb{E}\left((\beta V_k)^{-\ell m} \left[ \gamma^m_k(\alpha, \beta) - \mathbb{E}\left(\gamma_k^m(\alpha, \beta) \vert \bm{V} \right) \right]^\ell \middle\vert \bm{V} \right) \\
&= \left( 1 - \frac{t^2}{2} \left( \frac{\beta}{f(\beta; m,c)^2} \beta^{-2m} \mathbb{E}\left(\left[ \gamma^m_k(\alpha, \beta) -\mathbb{E}\left(\gamma^m_k(\alpha, \beta) \vert \bm{V} \right) \right]^2 \middle\vert \bm{V} \right) \right) + R(t) \right) \\
&= \left( 1 - \frac{t^2}{2} \frac{1}{\beta \left(\sum_{j=1}^m {m \brack j} \frac{\Gamma(j)}{c^{j-1}} \right)^2} \left( \frac{\Gamma(\beta V_k + 2m)}{\Gamma(\beta V_k)} - \left( \frac{\Gamma(\beta V_k + m)}{\Gamma(\beta V_k)} \right)^2 \right) + R(t) \right),
\eal
where the error term $R(t)$ is upper bounded by
\bal
&\min\left\{ \mathbb{E}\left( \left| t \frac{\sqrt{\beta}}{f(\beta; m,c)} \frac{\gamma^m_k(\alpha, \beta) - \mathbb{E}(\gamma^m_k(\alpha, \beta) \vert \bm{V})}{\beta^m} \right|^3 \middle\vert \bm{V} \right), \right. \\
&\qquad \qquad \qquad \left. 2\mathbb{E}\left( \left| t \frac{\sqrt{\beta}}{f(\beta; m,c)} \frac{\gamma^m_k(\alpha, \beta) - \mathbb{E}(\gamma^m_k(\alpha, \beta) \vert \bm{V})}{\beta^m}  \right|^2 \vert \bm{V} \right) \right\} \\
&\qquad = O\left( \frac{t^3 V_k^{3m-2}}{\sqrt{\beta} f(\beta; m,c)^3} \right)
\eal
since the second conditional moment is finite. 

Taking the product of the conditional characteristic functions gives
\bal
\prod_{k=1}^{\lfloor \alpha^2 \rfloor} \mathbb{E}\left( \exp\left( it \frac{\sqrt{\beta}}{f(\beta; m,c)} \frac{\gamma^m_k(\alpha, \beta) - \frac{\Gamma(\beta V_k + m)}{\Gamma(\beta V_k)}}{\beta^m}  \right) \middle\vert \bm{V} \right) = \prod_{k=1}^{\lfloor \alpha^2 \rfloor}(1 + c_{\alpha, \beta, k}),
\eal
where
\bal
c_{\alpha, \beta, k} &=  -\frac{t^2}{2} \frac{1}{\beta \left(\sum_{j=1}^m {m \brack j} \frac{\Gamma(j)}{c^{j-1}} \right)^2} \left( \frac{\Gamma(\beta V_k + 2m)}{\Gamma(\beta V_k)} - \left( \frac{\Gamma(\beta V_k + m)}{\Gamma(\beta V_k)} \right)^2 \right) \\
 &\qquad \qquad + O\left( \frac{t^3 V_k^{3m-2}}{\sqrt{\beta} f(\beta; m,c)^3} \right).
\eal

Next we show that $c_{\alpha, \beta, k}$ satisfies the conditions of Lemma \ref{expconv}.  Indeed, observe that
\bal
\frac{\Gamma(\beta V_k + 2m)}{\Gamma(\beta V_k)} &= \sum_{j=1}^{2m} {2m \brack j} \beta^j V_k^j \\
\left(\frac{\Gamma(\beta V_k + m)}{\Gamma(\beta V_k)}\right)^2 &= \left( \sum_{j=1}^{m} {m \brack j} \beta^j V_k^j \right)^2 = \sum_{1 \leq i, j \leq m} {m \brack i}{m \brack j} \beta^{i+j} V_k^{i+j}, 
\eal
so that
\bal
&\frac{1}{\beta \left(\sum_{j=1}^m {m \brack j} \frac{\Gamma(j)}{c^{j-1}} \right)^2} \left( \frac{\Gamma(\beta V_k + 2m)}{\Gamma(\beta V_k)} - \left( \frac{\Gamma(\beta V_k + m)}{\Gamma(\beta V_k)} \right)^2 \right) \\
&\qquad \qquad = \frac{ \sum_{j=1}^{2m} {2m \brack j} \beta^{j-1} V_k^j - \sum_{1 \leq i, j \leq m} {m \brack i}{m \brack j} \beta^{i+j-1} V_k^{i+j} }{\left( \sum_{j=1}^m {m \brack j} \frac{\Gamma(j)}{c^{j-1}} \right)^2}.
\eal
It follows from Lemma \ref{as-LLN} that almost surely, we have 
\bal
&\lim_{\alpha, \beta \to \infty, \alpha/\beta \to c} \sum_{k=1}^{\lfloor \alpha^2 \rfloor} c_{\alpha, \beta, k} \\
&= -\frac{t^2}{2} \frac{1}{\left( \sum_{j=1}^m {m \brack j} \frac{\Gamma(j)}{c^{j-1}} \right)^2} \lim_{\alpha, \beta \to \infty, \alpha/\beta\to c} \left[ \sum_{j=1}^{2m} {2m \brack j} \beta^{j-1} \frac{\Gamma(j)}{\alpha^{j-1}} \left(\frac{\alpha^{j-1}}{\Gamma(j)}\sum_{k=1}^{\lfloor \alpha^2 \rfloor} V_k^j \right) \right. \\
&\qquad \left. - \sum_{1 \leq i,j \leq m} {m \brack i}{m \brack j} \beta^{i+j-1}  \frac{\Gamma(i+j)}{\alpha^{i+j-1}} \left(\frac{\alpha^{i+j-1}}{\Gamma(i+j)}\sum_{k=1}^{\lfloor \alpha^2 \rfloor} V_k^{i+j} \right) \right. \\
&\qquad \left. + O\left( \frac{t^3 \frac{\Gamma(3m-2)}{\alpha^{3m-3}} \left(\frac{\alpha^{3m-3}}{\Gamma(3m-2)}\sum_{k=1}^\infty V_k^{3m-2} \right)}{\sqrt{\beta} f(\beta; m,c)^3} \right) \right] \\
&= -\frac{t^2}{2} \left( \frac{ \sum_{j=1}^{2m} {2m \brack j} \frac{\Gamma(j)}{c^{j-1}} - \sum_{1 \leq i,j \leq m} {m \brack i}{m \brack j} \frac{\Gamma(i+j)}{c^{i+j-1}}}{\left( \sum_{j=1}^m {m \brack j} \frac{\Gamma(j)}{c^{j-1}} \right)^2} \right).
\eal

It remains to show $\max_{1 \leq k \leq \lfloor \alpha^2 \rfloor} |c_{\alpha, \beta, k}| \to 0$ as $\alpha, \beta, \to \infty$. 
By Lemma \ref{Gammamoments}, we have that 
\bal
c_{\alpha, \beta, k} &= -\frac{t^2}{2}\left( \frac{\beta}{f(\beta; m,c)^2} \beta^{-2m} \mathbb{E}\left( \left[ \gamma^m_k(\alpha, \beta) - \frac{\Gamma(\beta V_k + m)}{\Gamma(\beta V_k)} \right]^2 \middle\vert \bm{V} \right) \right) \\
&\qquad \qquad + O\left( \frac{t^3 V_k^{3m-2}}{\sqrt{\beta} f(\beta; m,c)^3} \right) \\
&= -\frac{t^2}{2}\left( \frac{\beta V_k^{2m}}{f(\beta; m,c)^2} \left( \frac{m^2}{\beta V_k} + O\left( \frac{1}{\beta^2 V_k^2} \right) \right) \right) + O\left( \frac{t^3 V_k^{3m-2}}{\sqrt{\beta} f(\beta; m,c)^3} \right) \\
&= -\frac{t^2}{2}\left( \frac{m^2 V_k^{2m-1}}{f(\beta; m,c)^2}  + O\left( \frac{V_k^{2m-2}}{\beta f(\beta; m,c)^2} \right) \right) + O\left( \frac{t^3 V_k^{3m-2}}{\sqrt{\beta} f(\beta; m,c)^3} \right).
\eal
Next recall that $\sum_{k=1}^\infty \mathbb{E}[V_k^j] = \frac{\Gamma(j)}{(\alpha+1)_{(j-1)}}$. In particular, we have that 
\bal
\frac{\sum_{k=1}^\infty \mathbb{E}[V_k^{2m-1}]}{f(\beta; m,c)^2} = \frac{\frac{\Gamma(2m-1)}{\alpha^{2m-2}}}{\frac{1}{\beta^{2m-2}} \left( \sum_{j=1}^m {m \brack j} \frac{\Gamma(j)}{c^{j-1}} \right)^2} \sim \frac{\Gamma(2m-1)}{c^{2m-2} \left( \sum_{j=1}^m {m \brack j} \frac{\Gamma(j)}{c^{j-1}} \right)^2},
\eal
For any $\epsilon > 0$ and $N \geq 1$, we have that
\bal
&P\left( \frac{\max_{1 \leq k \leq \lfloor \alpha^2 \rfloor} V_k^{2m-1}}{\sum_{k=1}^\infty \mathbb{E}\left[V_k^{2m-1}\right]} > \epsilon \right) 
\leq \sum_{k=1}^{\lfloor \alpha^2\rfloor} P\left( \frac{V_k^{2m-1}}{\sum_{k=1}^\infty \mathbb{E}\left[V_k^{2m-1}\right]} > \epsilon \right) \\
&= \sum_{k=1}^{\lfloor \alpha^2 \rfloor} P\left( V_k^N > \left( \epsilon \sum_{k=1}^\infty \mathbb{E}\left[V_k^{2m-1} \right] \right)^{\frac{N}{2m-1}} \right) \\
&\leq \frac{1}{\left( \epsilon \sum_{k=1}^\infty \mathbb{E}\left[V_k^{2m-1}\right] \right)^{\frac{N}{2m-1}}} \sum_{k=1}^{\lfloor \alpha^2 \rfloor} \mathbb{E}[V_k^N] \\
&= \frac{1}{\left( \epsilon \sum_{k=1}^\infty \mathbb{E}\left[V_k^{2m-1}\right] \right)^{\frac{N}{2m-1}}} \frac{N!}{(\alpha + 1)_{(N)}} \frac{\alpha + N}{N} \left(1 - \left( \frac{\alpha}{\alpha + N}\right)^{\lfloor \alpha^2 \rfloor} \right) \\
&\leq \frac{1}{\left( \epsilon \sum_{k=1}^\infty \mathbb{E}\left[V_k^{2m-1}\right]\right)^{\frac{N}{2m-1}}} \frac{\Gamma(N)}{(\alpha + 1)_{(N-1)}}  \\
&\sim \left(\frac{\epsilon \Gamma(2m-1)}{\alpha^{2m-2}} \right)^{-\frac{N}{2m-1}} \frac{\Gamma(N)}{\alpha^{N-1}}
\eal
Choosing $N = 3(2m-1)$ gives
\bal
P\left( \frac{\max_{1 \leq k \leq \lfloor \alpha^2 \rfloor} V_k^{2m-1}}{\sum_{k=1}^\infty \mathbb{E}\left[V_k^{2m-1}\right]} > \epsilon \right) &\leq \frac{1}{\alpha^2} \frac{\Gamma(6m-3)}{\left( \epsilon \Gamma(2m-1) \right)^3},
\eal
so that 
\bal
\sum_{\alpha = 1}^\infty P\left( \frac{\max_{1 \leq k \leq \lfloor \alpha^2 \rfloor} V_k^{2m-1}}{\sum_{k=1}^\infty \mathbb{E}\left[V_k^{2m-1}\right]} > \epsilon \right) &\leq \sum_{\alpha = 1}^\infty \frac{1}{\alpha^2} \frac{\Gamma(6m-3)}{\left( \epsilon \Gamma(2m-1) \right)^3} < \infty. 
\eal
Thus by Borel-Cantelli, we have that $\frac{\max_{1 \leq k \leq \lfloor \alpha^2 \rfloor} V_k^{2m-1}}{\sum_{k=1}^\infty \mathbb{E}[V_k^{2m-1}]} \to 0$ as $\alpha, \beta \to \infty$. Similarly we can show that 
\bal
\frac{\max_{1 \leq k \leq \lfloor \alpha^2 \rfloor}V_k^{2m-2}}{\sum_{k=1}^\infty E(V_k^{2m-2})} \to 0 \quad \text{and} \quad
\frac{\max_{1\leq k\leq \lfloor \alpha^2 \rfloor} V_k^{3m-2}}{\sum_{k=1}^\infty E(V_k^{3m-2})} \to 0
\eal
as $\alpha, \beta \to \infty$, and that both convergences are uniform. 
Combining the above we have that $\max_{1 \leq k \leq \lfloor \alpha^2 \rfloor} |c_{\alpha, \beta, k}| \to 0$ as $\alpha, \beta \to \infty$.  

Therefore the conditional characteristic function converges as
\bal
\mathbb{E}\left( \exp\left( it \frac{\sqrt{\beta}}{f(\beta; m,c)} \sum_{k=1}^{\lfloor \alpha^2 \rfloor} \frac{\gamma^m_k(\alpha, \beta) - \frac{\Gamma(\beta V_k + m)}{\Gamma(\beta V_k)}}{\beta^m} \right) \middle\vert \bm{V} \right) \to \exp\left( -\frac{t^2 \sigma_2^2}{2}\right),
\eal
where
\bal
\sigma_2^2 = \frac{ \sum_{j=1}^{2m} {2m \brack j} \frac{\Gamma(j)}{c^{j-1}} - \sum_{1 \leq i,j \leq m} {m \brack i}{m \brack j} \frac{\Gamma(i+j)}{c^{i+j-1}}}{\left( \sum_{j=1}^m {m \brack j} \frac{\Gamma(j)}{c^{j-1}} \right)^2}.
\eal
Finally observe that
\bal
&\left| \mathbb{E}\left( \exp\left( it \frac{\sqrt{\beta}}{f(\beta; m,c)} \sum_{k=1}^{\lfloor \alpha^2 \rfloor} \frac{\gamma^m_k(\alpha, \beta) - \frac{\Gamma(\beta V_k + m)}{\Gamma(\beta V_k)}}{\beta^m} \right) \right) - \exp\left( -\frac{t^2 \sigma_2^2}{2}\right) \right| \\
&= \left| \mathbb{E}\left[ \mathbb{E}\left( \exp\left( it \frac{\sqrt{\beta}}{f(\beta; m,c)} \sum_{k=1}^{\lfloor \alpha^2 \rfloor} \frac{\gamma^m_k(\alpha, \beta) - \frac{\Gamma(\beta V_k + m)}{\Gamma(\beta V_k)}}{\beta^m} \right) \middle\vert \bm{V} \right)- \exp\left( -\frac{t^2 \sigma_2^2}{2}\right) \right] \right| \\
&\leq \mathbb{E}\left( \left| \mathbb{E}\left( \exp\left( it \frac{\sqrt{\beta}}{f(\beta; m,c)} \sum_{k=1}^{\lfloor \alpha^2\rfloor} \frac{\gamma^m_k(\alpha, \beta)- \frac{\Gamma(\beta V_k + m)}{\Gamma(\beta V_k)}}{\beta^m} \right) \middle\vert \bm{V} \right) - \exp\left( -\frac{t^2 \sigma_2^2}{2}\right) \right| \right) \\
&\to 0,
\eal
from which it follows that 
\bal
\mathbb{E}\left( \exp\left( it \frac{\sqrt{\beta}}{f(\beta; m,c)} \sum_{k=1}^{\lfloor \alpha^2 \rfloor} \frac{\gamma^m_k(\alpha, \beta) - \frac{\Gamma(\beta V_k + m)}{\Gamma(\beta V_k)}}{\beta^m} \right) \right) \to \exp\left( -\frac{t^2 \sigma_2^2}{2}\right).
\eal
Therefore $X_m(\alpha, \beta) \xrightarrow{D} N(0, \sigma_2^2)$ as $\alpha, \beta \to \infty$ such that $\frac{\alpha}{\beta} \to c$. 
\end{proof} 

\subsection{Joint Convergence} \label{SecJointConvergence}

Let 
\bal
Z(\alpha, \beta) &= \sqrt{\beta} \left(\sum_{k=1}^\infty \frac{\gamma_k(\alpha, \beta)}{\beta} - 1 \right)\\
\hat{Z}(\alpha, \beta) &= \sqrt{\beta} \sum_{k=1}^{\lfloor \alpha^2 \rfloor} \left( \frac{\gamma_k(\alpha, \beta)}{\beta} - V_k \right).
\eal
We will need the following joint convergence result. The proof is similar to the characteristic function approach of Lemma \ref{Level2CLTHDP}. 

\begin{lemma} \label{JointCLT}
We have the following multivariate central limit theorem
\bal
(Z(\alpha, \beta), X_m(\alpha, \beta)) \xrightarrow{D} N(\bm{0}, \bm{\Sigma})
\eal
as $\alpha,\beta \to \infty$ such that $\frac{\alpha}{\beta} \to c$, where the covariance matrix is given by
\bal
\bm{\Sigma} = \begin{pmatrix} 1 & m \\ m & \sigma_2^2 \end{pmatrix} 
\eal
and $\sigma_2^2$ is the level 2 variance from Lemma~\ref{Level2CLTHDP}.
\end{lemma}

\begin{proof} It is sufficient to show the joint convergence of $(\hat{Z}(\alpha, \beta), X_{m1}(\alpha, \beta))$, where $X_{m1}(\alpha, \beta)$ is the truncated level 2 contribution defined in the proof of Lemma \ref{Level2CLTHDP}. 

Let $(r,s) \in \R^2$. We compute the characteristic function of $(r,s) \cdot (\hat{Z}(\alpha, \beta), X_{m1}(\alpha, \beta)) = r \hat{Z}(\alpha, \beta, t) + s X_{m1}(\alpha, \beta)$. Conditioning on $\bm{V}$, we can write the characteristic function as
\bal
&\mathbb{E}\left( \exp\left( ir \hat{Z}(\alpha,\beta)+ is X_{m1}(\alpha,\beta)\right)\right) \\
&= \mathbb{E}\left[\mathbb{E}\left( \exp\left( ir \sqrt{\beta} \sum_{k=1}^{\lfloor \alpha^2 \rfloor} \left( \frac{\gamma_k(\alpha, \beta)}{\beta} - V_k \right) \right. \right. \right. \\
&\qquad \qquad \left. \left. \left. + is \frac{\sqrt{\beta}}{f(\beta; m,c)} \sum_{k=1}^{\lfloor \alpha^2 \rfloor} \frac{\gamma^m_k(\alpha, \beta) - \frac{\Gamma(\beta V_k + m)}{\Gamma(\beta V_k)}}{\beta^m} \right) \middle\vert \bm{V} \right) \right] \\
&= \mathbb{E}\left[ \prod_{k=1}^{\lfloor \alpha^2 \rfloor} \mathbb{E}\left( \exp\left( ir \sqrt{\beta} \left( \frac{\gamma_k(\alpha, \beta)}{\beta} - V_k  \right) + is \frac{\sqrt{\beta}}{f(\beta; m,c)} \frac{\gamma^m_k(\alpha, \beta) - \frac{\Gamma(\beta V_k + m)}{\Gamma(\beta V_k)}}{\beta^m} \right) \middle\vert \bm{V} \right) \right]. 
\eal

The inner conditional characteristic function is computed to be
\bal
& \mathbb{E}\left( \sum_{\ell = 0}^\infty \frac{(i\sqrt{\beta})^\ell}{\ell!} \left[ r \left( \frac{\gamma_k(\alpha, \beta)}{\beta} - V_k \right) + \frac{s}{f(\beta; m,c)} \frac{\gamma^m_k(\alpha, \beta) - \frac{\Gamma(\beta V_k + m)}{\Gamma(\beta V_k)}}{\beta^m} \right]^\ell \middle\vert \bm{V} \right) \\
&= \sum_{\ell = 0}^\infty \frac{(i\sqrt{\beta})^\ell}{\ell!} \mathbb{E}\left( \left[ r \left( \frac{\gamma_k(\alpha, \beta)}{\beta} - V_k \right) + \frac{s}{f(\beta; m,c)} \frac{\gamma^m_k(\alpha, \beta) - \frac{\Gamma(\beta V_k + m)}{\Gamma(\beta V_k)}}{\beta^m} \right]^\ell \middle\vert \bm{V} \right) \\
&= 1 - \frac{\beta}{2} \mathbb{E}\left(\left[ r \left( \frac{\gamma_k(\alpha, \beta)}{\beta} - V_k \right) + \frac{s}{f(\beta; m,c)} \frac{\gamma^m_k(\alpha, \beta) - \frac{\Gamma(\beta V_k + m)}{\Gamma(\beta V_k)}}{\beta^m} \right]^2 \middle\vert \bm{V} \right) + R(r,s),
\eal
where the error term $R(r,s) \to 0$ as $\alpha, \beta \to \infty$. 

Expanding the second moment gives
\bal
&\mathbb{E}\left(\left[ r \left( \frac{\gamma_k(\alpha, \beta)}{\beta} - V_k \right) + \frac{s}{f(\beta; m,c)} \frac{\gamma^m_k(\alpha, \beta) - \frac{\Gamma(\beta V_k + m)}{\Gamma(\beta V_k)}}{\beta^m} \right]^2 \middle\vert \bm{V} \right) \\
&= r^2 \mathbb{E}\left( \left( \frac{\gamma_k(\alpha, \beta )}{\beta} - V_k \right)^2 \middle\vert \bm{V} \right) + \frac{s^2}{f(\beta; m,c)^2} \mathbb{E}\left( \left( \frac{\gamma^m_k(\alpha, \beta) - \frac{\Gamma(\beta V_k + m)}{\Gamma(\beta V_k)}}{\beta^m} \right)^2 \middle\vert \bm{V} \right) \\
&\qquad + \frac{2rs}{f(\beta; m,c)} \mathbb{E}\left(\left( \frac{\gamma_k(\alpha, \beta)}{\beta} - V_k \right)\left( \frac{\gamma^m_k(\alpha, \beta) - \frac{\Gamma(\beta V_k + m)}{\Gamma(\beta V_k)}}{\beta^m} \right) \middle\vert \bm{V} \right). 
\eal
The expectations in the first and second summands follow from Lemma \ref{Gammamoments},
\bal
\mathbb{E}\left( \left( \frac{\gamma_k(\alpha, \beta)}{\beta} - V_k \right)^2 \middle\vert \bm{V} \right) &= \beta^{-2}\left[ \frac{\Gamma(\beta V_k + 2)}{\Gamma(\beta V_k)} - \left( \frac{\Gamma(\beta V_k + 1)}{\Gamma(\beta V_k)} \right)^2 \right] = \frac{V_k}{\beta} \\
\mathbb{E}\left( \left( \frac{\gamma^m_k(\alpha, \beta) - \frac{\Gamma(\beta V_k + m)}{\Gamma(\beta V_k)}}{\beta^m} \right)^2 \middle\vert \bm{V} \right) &= \beta^{-2m} \left[ \frac{\Gamma(\beta V_k + 2m)}{\Gamma(\beta V_k)} - \left( \frac{\Gamma(\beta V_k + m)}{\Gamma(\beta V_k)} \right)^2 \right].
\eal 
The expectation in the third (cross) summand can be expanded as
\bal
&\mathbb{E}\left(\left( \frac{\gamma_k(\alpha, \beta)}{\beta} - V_k \right)\left( \frac{\gamma^m_k(\alpha, \beta) - \frac{\Gamma(\beta V_k + m)}{\Gamma(\beta V_k)}}{\beta^m} \right) \middle\vert \bm{V} \right) \\
&= \frac{1}{\beta^{m+1}} \E\left[\left( \gamma_k(\alpha, \beta) - \mathbb{E}[\gamma_k(\alpha, \beta) \mid \bm{V}]  \right)\left( \gamma^m_k(\alpha, \beta) - \mathbb{E}[\gamma^m_k(\alpha, \beta) \mid \bm{V}]  \right) \vert \bm{V} \right] \\
&= \frac{\mathbb{E}[\gamma^{m+1}_k(\alpha, \beta) \mid \bm{V}] - \mathbb{E}[\gamma_k(\alpha, \beta) \mid \bm{V}] \mathbb{E}[\gamma^m_k(\alpha, \beta) \mid \bm{V}]}{\beta^{m+1}} \\ 
&= \frac{\Gamma(\beta V_k + m + 1) - \beta V_k \Gamma(\beta V_k + m)}{\beta^{m+1} \Gamma(\beta V_k)} \\
&= \frac{m \Gamma(\beta V_k + m)}{\beta^{m+1} \Gamma(\beta V_k)}.
\eal
Thus the conditional characteristic function is
\bal
&1 - \frac{\beta}{2} \mathbb{E}\left(\left[ r \left( \frac{\gamma_k(\alpha, \beta)}{\beta} - V_k \right) + \frac{s}{f(\beta; m,c)} \frac{\gamma^m_k(\alpha, \beta) - \frac{\Gamma(\beta V_k + m)}{\Gamma(\beta V_k)}}{\beta^m} \right]^2 \middle\vert \bm{V} \right) + R(r,s) \\
&= 1 - \frac{1}{2} \left\{ r^2V_k + \frac{2rs m\Gamma(\beta V_k + m)}{f(\beta; m,c) \beta^m \Gamma(\beta V_k)} \right. \\
& \qquad \left. + \frac{s^2}{\beta \left( \sum_{j=1}^m {m \brack j} \frac{\Gamma(j)}{c^{j-1}} \right)^2}  \left[ \frac{\Gamma(\beta V_k + 2m)}{\Gamma(\beta V_k)} - \left( \frac{\Gamma(\beta V_k + m)}{\Gamma(\beta V_k)} \right)^2 \right]  \right\} + R(r,s). 
\eal
Recall that with probability one, we have that
\bal
&\sum_{k=1}^\infty V_k = 1, \\ 
&\frac{\alpha^{j-1}}{\Gamma(j)} \sum_{k=1}^\infty V_k^j \to 1,\\ 
&\frac{ 1 }{\beta \left( \sum_{j=1}^m {m \brack j} \frac{\Gamma(j)}{c^{j-1}} \right)^2} \sum_{k=1}^\infty \left[ \frac{\Gamma(\beta V_k + 2m)}{\Gamma(\beta V_k)} - \left( \frac{\Gamma(\beta V_k + m)}{\Gamma(\beta V_k)} \right)^2 \right] \to \sigma_2^2,
\eal 
and observe that
\bal
\frac{1}{f(\beta; m,c)} \sum_{k=1}^\infty \frac{\Gamma(\beta V_k + m)}{\beta^m \Gamma(\beta V_k)} &= \frac{1}{f(\beta; m,c) \beta^{m-1}} \sum_{j=1}^m {m \brack j} \beta^{j-1} \sum_{k=1}^\infty V_k^j \\
&= \frac{1}{\sum_{j=1}^m {m \brack j} \frac{\Gamma(j)}{c^{j-1}}} \sum_{j=1}^m {m \brack j} \beta^{j-1} \frac{\Gamma(j)}{\alpha^{j-1}} \left( \frac{\alpha^{j-1}}{\Gamma(j)} \sum_{k=1}^\infty V_k^j \right) \\
&\to 1
\eal
almost surely as $\alpha, \beta \to \infty$ such that $\frac{\alpha}{\beta} \to c$, where $\sigma_2^2$ is the level 2 variance from Lemma \ref{Level2CLTHDP}. 

Therefore by Lemma \ref{expconv}, the product of conditional characteristic functions converges as
\bal
& \prod_{k=1}^{\lfloor \alpha^2\rfloor} \left\{ 1 - \frac{1}{2} \left[ r^2V_k + \frac{s^2}{\beta \left( \sum_{j=1}^m {m \brack j} \frac{\Gamma(j)}{c^{j-1}} \right)^2}  \left[ \frac{\Gamma(\beta V_k + 2m)}{\Gamma(\beta V_k)} - \left( \frac{\Gamma(\beta V_k + m)}{\Gamma(\beta V_k)} \right)^2 \right] \right. \right. \\
&\qquad \qquad \left. \left. + \frac{2rs m\Gamma(\beta V_k + m)}{f(\beta; m,c) \beta^m \Gamma(\beta V_k)} \right] + R(r,s) \right\} \\
&\to \exp\left( -\frac{1}{^2} \left( r^2 + s^2 \sigma_2^2+ 2rsm \right) \right).
\eal
It follows that 
\bal
&\mathbb{E}\left( \exp\left( ir \sqrt{\beta} \sum_{k=1}^{\lfloor \alpha^2 \rfloor} \left( \frac{\gamma_k(\alpha, \beta)}{\beta} - V_k \right) + is \frac{\beta^{m-1/2}}{\sum_{j=1}^m {m \brack j} \frac{\Gamma(j)}{c^{j-1}}} \sum_{k=1}^{\lfloor \alpha^2 \rfloor} \frac{\gamma^m_k(\alpha, \beta) - \frac{\Gamma(\beta V_k + m)}{\Gamma(\beta V_k)}}{\beta^m} \right) \right) \\
&\to \exp\left( -\frac{1}{^2} \left( r^2 + s^2 \sigma_2^2 + 2rsm \right) \right)
\eal
as $\alpha, \beta \to \infty$ such that $\frac{\alpha}{\beta} \to c$, which implies the result.  
\end{proof}

\subsection{Proof of the CLT for the Homozygosity of the HDP} \label{SecProofofCLTHDP}

We now prove the central limit theorem for the homozygosity of the hierarchical Dirichlet process. 

\begin{proof}[Proof of Theorem \ref{CLTHomozygosityHDP}]
Recall that for $m \geq 2$, we can decompose $\tilde{H}(\alpha, \beta)$ as 
\bal
\tilde{H}_m(\alpha,\beta) &= X_m(\alpha, \beta) + Y_m(\alpha, \beta) \\
&+ \sqrt{\beta} \left( \left( \frac{\beta}{\gamma(\beta)} \right)^m - 1 \right) \frac{\sum_{k=1}^\infty\frac{\gamma_k^m(\alpha, \beta)}{\beta^m}}{f(\beta; m,c)}.
\eal
By Lemmas \ref{Level1CLTHDP} and \ref{Level2CLTHDP},
\bal
X_m(\alpha, \beta) &\xrightarrow{D} X \sim N(0,\sigma_2^2) \\
Y_m(\alpha, \beta) &\xrightarrow{D} Y \sim N\left(0, \sigma_1^2 \right)
\eal
as $\alpha, \beta \to \infty$ such that $\frac{\alpha}{\beta} \to c$. By the law of large numbers, 
\bal
\frac{\sum_{k=1}^\infty\frac{\gamma_k^m(\alpha, \beta)}{\beta^m}}{f(\beta; m,c)} = \frac{\sum_{k=1}^\infty \frac{\gamma^m_k(\alpha, \beta)}{\beta^m}}{\frac{1}{\beta^{m-1}}\sum_{j=1}^m {m \brack j} \frac{\Gamma(j)}{c^{j-1}}} \xrightarrow{D} 1
\eal
as $\alpha, \beta \to \infty$ such that $\frac{\alpha}{\beta} \to c$. Next, by Lemma \ref{Fen7.7}, 
\bal
Z(\alpha, \beta) = \sqrt{\beta} \left( \sum_{k=1}^\infty \frac{\gamma_k(\alpha, \beta)}{\beta} - 1 \right) = \sqrt{\beta} \left( \left( \frac{\gamma(\beta)}{\beta} \right) - 1 \right) \xrightarrow{D} Z \sim N(0,1)
\eal
as $\beta \to \infty$, so that
\bal
\sqrt{\beta}\left( \left( \frac{\beta}{\gamma(\beta)} \right)^m - 1 \right) =  \sqrt{\beta} \left( \left( \frac{\gamma(\beta)}{\beta} \right)^{-m} - 1 \right) \xrightarrow{D} -mZ.
\eal
Thus by Slutsky's theorem,
\bal
&\sqrt{\beta} \left( \left( \frac{\beta}{\gamma(\beta)} \right)^m - 1 \right) \frac{\sum_{k=1}^\infty\frac{\gamma_k^m(\alpha, \beta)}{\beta^m}}{f(\beta; m,c)} \xrightarrow{D} -mZ \\
\eal
as $\alpha, \beta \to \infty$ such that $\frac{\alpha}{\beta} \to c$. 

By Lemma \ref{JointCLT}, $(Z(\alpha, \beta), X_m(\alpha, \beta)) \xrightarrow{D} N(\bm{0}, \bm{\Sigma})$, where
\bal
\bm{\Sigma} = \begin{pmatrix} 1 & m \\ m & \sigma_2^2 \end{pmatrix}. 
\eal
Finally $Y_m(\alpha, \beta)$ is independent of $Z(\alpha, \beta)$, and $(X_m(\alpha, \beta), Y_m(\alpha, \beta)) \xrightarrow{D} N\left(\bm{0}, \begin{pmatrix} \sigma_2^2 & 0 \\ 0 & \sigma_1^2 \end{pmatrix} \right)$ by a similar argument as the proof of Lemma \ref{JointCLT}. 

Therefore for all $m \geq 2$, the the scaled homozygosity converges in distribution as
\bal
\tilde{H}_m(\alpha, \beta) \xrightarrow{D} X + Y - mZ
\eal
as $\alpha, \beta \to \infty$ such that $\frac{\alpha}{\beta} \to c$, which is distributed as a normal random variable, $N(0,\sigma_c^2)$, with variance
\bal
\sigma_c^2 &= \Var(X) + \Var(Y) + m^2\Var(Z) - 2m\cov(X,Z) \\
&= \sigma_1^2 + \sigma_2^2 - m^2 \\
&= \frac{\sum_{1 \leq i,j \leq m}  {m \brack i}{m \brack j} \frac{\Gamma(i+j) - \Gamma(i+1)\Gamma(j+1)}{c^{i+j-1}}}{\left( \sum_{j=1}^m {m \brack j} \frac{\Gamma(j)}{c^{j-1}} \right)^2}\\
&\qquad + \frac{ \sum_{j=1}^{2m} {2m \brack j} \frac{\Gamma(j)}{c^{j-1}} - \sum_{1 \leq i,j \leq m} {m \brack i}{m \brack j} \frac{\Gamma(i+j)}{c^{i+j-1}}}{\left( \sum_{j=1}^m {m \brack j} \frac{\Gamma(j)}{c^{j-1}} \right)^2} - m^2 \\
&= \frac{\sum_{j=1}^{2m} {2m \brack j} \frac{\Gamma(j)}{c^{j-1}} - \sum_{1 \leq i,j \leq m}  {m \brack i}{m \brack j} \frac{\Gamma(i+1)\Gamma(j+1)}{c^{i+j-1}}}{\left( \sum_{j=1}^m {m \brack j} \frac{\Gamma(j)}{c^{j-1}} \right)^2} - m^2.  \qedhere
\eal 
\end{proof}

\section{Central Limit Theorem for the Homozygosity of the FDHDP} \label{CLTHomozygFDHDP}

Recall that in the FDHDP, the level 1 random measure becomes
\bal
\Xi^n_{\alpha, \nu}= \sum_{k=1}^n W_{nk} \delta_{\xi_k} = \sum_{k=1}^n \frac{\gamma_k}{\gamma} \delta_{\xi_k},
\eal
where $(W_{n1},\ldots, W_{nn}) \sim \Dir\left( \frac{\alpha}{n}, \ldots, \frac{\alpha}{n} \right)$, and independently, $\{\xi_k\} $ are independent and identically distributed with common distribution $\nu$. The FDHDP is the random measure 
\bal
\Xi^n_{\alpha,\beta, \nu} =\sum_{i=1}^n Z_{ni}\delta_{\xi_i},
\eal
where $\bm{Z}_n = (Z_{n1}, \ldots, Z_{nn}) \sim \Dir(\beta W_{n1}, \ldots, \beta W_{nn})$. The corresponding $m$th order homozygosity is
\bal
H_{m, n}(\alpha,\beta)= \sum_{i=1}^n Z_{ni}^m 
\eal
for $m \geq 2$.

In this section we prove Theorem \ref{CLTHomozygosityFDHDP}, the central limit theorem for the homozygosity of the FDHDP. The differences in the asymptotic behaviors between the HDP and the FDHDP are due to the different asymptotic behaviors of $\Xi_{\alpha,\nu}$ and $\Xi^n_{\alpha,\nu}$ for large $\alpha$ and $n$. For fixed $d$, this is seen from the following diagram:

 $$\label{dia1}
 \begin{tikzcd}[row sep=14pt,column sep={18mm,between origins}]
&\Xi^n_{\alpha,\nu}  \arrow[dddd,"\begin{matrix}\alpha\vspace{-1mm}\\ \vspace{-1mm}\big\downarrow \\\vspace{-2mm} \infty\end{matrix}"] \arrow[ddddrrr, "\alpha\sim d n\rightarrow \infty"]  \arrow[rrr,"n\rightarrow\infty"] & &  &\ \ \Xi_{\alpha,\nu}\arrow[dddd,"\begin{matrix}\alpha\vspace{-1mm}\\ \vspace{-1mm}\big\downarrow \\ \vspace{-2mm} \infty\end{matrix}"] & \\
 &  & &  &  & \\
    &  & &  &  & \\[4mm]
 &  & &  &  & \\
    & \mbox{empirical law} \arrow[rrr,"n\rightarrow \infty"] & &  &\ \ \ \nu& \\
\end{tikzcd}
$$

In the central limit theorem for the FDHDP, the level 1 asymptotic corresponds to the convergence along the diagonal. Since several parts of the proof are similar to the HDP case, we will simply list our results, and only include proofs which are significantly different.

\subsection{Gamma Representation, Law of Large Numbers, and Decomposition for the FDHDP}

Let  
\bal
\gamma_k(\alpha,\beta, n) = \gamma\left(\sum_{i=1}^k W_{ni} \right) - \gamma\left(\sum_{i=1}^{k-1}W_{ni}\right)
\eal
for all $k=1, 2, \ldots, n$. Then we have that
\bal
Z_{nk} \stackrel{D}{=} \frac{\gamma_k(\alpha,\beta, n)}{\gamma(\beta)}
\eal
for all $k=1, 2, \ldots, n$. 

The following lemma gives the expectation of the homozygosity for the FDHDP. 

\begin{lemma} \label{weightmeans}
For any $m \geq 2$,
\bal
\mathbb{E}[Z_{nk}^m] = \frac{1}{(\beta)_{(m)}} \sum_{j=1}^m {m \brack j} \beta^{j} \frac{\left( \frac{\alpha}{n} \right)_{(j)}}{(\alpha)_{(j)}}
\eal
and 
\bal
\mathbb{E}[H_{m,n}(\alpha,\beta)] = \frac{1}{(\beta)_{(m)}} \sum_{j=1}^m {m \brack j} \beta^{j} \frac{\left( \frac{\alpha}{n} + 1 \right)_{(j-1)}}{(\alpha + 1)_{(j-1)}}.
\eal
\end{lemma}

Observe that 
\bal
\mathbb{E}[H_{m,n}(\alpha,\beta)] \asymp \frac{1}{\beta^{m-1}} \sum_{j=1}^m {m \brack j} \frac{\beta^{j-1} \left( \frac{\alpha}{n} + 1 \right)_{(j-1)}}{(\alpha+1)_{(j-1)}} \sim \frac{1}{\beta^{m-1}} \sum_{j=1}^m {m \brack j} \frac{(d+1)_{(j-1)}}{c^{j-1}}.
\eal
For notational simplicity, let
\bal
f(\beta; m,c,d) = \frac{1}{\beta^{m-1}} \sum_{j=1}^m {m \brack j} \frac{(d+1)_{(j-1)}}{c^{j-1}}.
\eal
We have the following law of large numbers for $H_{m,n}(\alpha,\beta)$. 

\begin{proposition}\label{LLNFDHDP}
The homozygosity of the FDHDP satisfies
\bal
\frac{H_{m,n}(\alpha,\beta)}{f(\beta;m,c,d)} \to 1
\eal
in probability as $\alpha, \beta, n \to \infty$ such that $\frac{\alpha}{\beta} \to c$ and $\frac{\alpha}{n} \to d$. 
\end{proposition}

Next we give a decomposition for the scaled homozygosity $\tilde{H}_{m,n}(\alpha, \beta)$. Let 
\bal
X_m(\alpha, \beta, n) &= \frac{\sqrt{\beta}}{f(\beta; m,c,d)} \sum_{k=1}^n \frac{\gamma^m_k(\alpha, \beta, n) - \frac{\Gamma(\beta W_{nk} + m)}{\Gamma(\beta W_{nk})}}{\beta^m}, \\
Y_m(\alpha, \beta, n) &= \frac{\sqrt{\beta}}{f(\beta; m,c,d)} \left( \sum_{k=1}^n \frac{\Gamma(\beta W_{nk} + m)}{\beta^m \Gamma(\beta W_{nk})} - \frac{1}{\beta^{m-1}} \sum_{j=1}^m {m \brack j} \frac{\beta^{j-1} \left( \frac{\alpha}{n} + 1 \right)_{(j-1)}}{(\alpha+1)_{(j-1)}} \right),
\eal
and observe that for all $m \geq 2$, 
\bal
\tilde{H}_{m,n}(\alpha,\beta) &= X_m(\alpha,\beta, n) + Y_m(\alpha, \beta, n) + \sqrt{\beta} \left( \left( \frac{\beta}{\gamma(\beta)} \right)^m - 1 \right) \frac{\sum_{k=1}^\infty\frac{\gamma_k^m(\alpha, \beta, n)}{\beta^m}}{f(\beta; m,c, d)}.
\eal
We again refer to $X_m(\alpha,\beta,n)$ as the {\em level 2 contribution} and $Y_m(\alpha,\beta,n)$ as the {\em level 1 contribution}. 

\subsection{Level 2 Contribution for the FDHDP}

In this section, we give the central limit theorem for the level 2 contribution, $X_m(\alpha, \beta, n)$. The proof follows similarly as in the full HDP case using characteristic functions, and so we omit the details. 

\begin{lemma} \label{Level2CLTFDHDP}
The level 2 contribution satisfies
\bal
X_m(\alpha, \beta, n) = \frac{\sqrt{\beta}}{f(\beta; m,c,d)} \sum_{k=1}^n \frac{\gamma^m_k(\alpha, \beta, n) - \frac{\Gamma(\beta W_{nk} + m)}{\Gamma(\beta W_{nk})}}{\beta^m} \xrightarrow{D} N\left(0,v^2_2(c,d)\right)
\eal
as $\alpha, \beta, n \to \infty$ such that $\frac{\alpha}{\beta} \to c$ and $\frac{\alpha}{n} \to d$, where the variance is given by
\bal
v^2_2(c,d) = \frac{\sum_{j=1}^{2m} {2m \brack j} \frac{(d+1)_{(j-1)}}{c^{j-1}} - \sum_{1 \leq i,j \leq m} {m \brack i}{m \brack j} \frac{(d+1)_{(i+j-1)}}{c^{i+j-1}}}{\left( \sum_{j=1}^m {m \brack j} \frac{(d+1)_{(j-1)}}{c^{j-1}} \right)^2}.
\eal 
\end{lemma}

\subsection{Level 1 Contribution for the FDHDP}

In this section we prove the central limit theorem for the level 1 contribution, $Y_m(\alpha, \beta, n)$. First observe that we can rewrite the level 1 contribution as
\bal
Y_m(\alpha,\beta,n) &= \frac{\sqrt{\beta}}{f(\beta; m,c,d)} \left( \sum_{k=1}^n \frac{\Gamma(\beta W_{nk} + m)}{\beta^m \Gamma(\beta W_{nk})} - \frac{1}{\beta^{m-1}} \sum_{j=1}^m {m \brack j} \frac{\beta^{j-1} \left( \frac{\alpha}{n} + 1 \right)_{(j-1)}}{(\alpha+1)_{(j-1)}} \right)  \\
&= \frac{1}{\sum_{j=1}^m {m \brack j} \frac{(d+1)_{(j-1)}}{c^{j-1}}} \sum_{j=2}^m {m \brack j} \left( \frac{\beta}{n} \right)^{j - \frac{1}{2}} n^{j-\frac{1}{2}} \sum_{k=1}^n \left( W_{nk}^j - \mathbb{E}\left(W_{nk}^j\right) \right). 
\eal
Then the level 1 contribution satisfies a central limit theorem if we can show that the vector
\bal
\left( n^{3/2} \sum_{k=1}^n (W_{nk}^2 - \mathbb{E}(W_{nk}^2)),  \ldots, n^{m-\frac{1}{2}} \sum_{k=1}^n (W_{nk}^m - \mathbb{E}(W_{nk}^m)) \right)
\eal
converges in distribution to a multivariate normal. 

Recall that we can write $\sum_{k=1}^n W_{nk}^j = \frac{\sum_{k=1}^n Y_k^j}{\left( \sum_{k=1}^n Y_k \right)^j}$, where $\{Y_k\}$ are iid $\Gam\left( \frac{\alpha}{n}, 1 \right)$ random variables. We start with the following multivariate central limit theorem. 

\begin{lemma}\label{jointgammaclt}
Let $\{Y_k\}$ be iid $\Gam\left( \frac{\alpha}{n}, 1 \right)$ random variables. Then 
\bal
\left( \frac{1}{\sqrt{n}} \sum_{k=1}^n (Y_k - \mathbb{E}(Y_1)), \ldots, \frac{1}{\sqrt{n}} \sum_{k=1}^n (Y_k^m - \mathbb{E}(Y_1^m)) \right) \xrightarrow{D} N(\bm{0}, \bm{\Sigma})
\eal
as $\alpha, n \to \infty$ such that $\frac{\alpha}{n} \to d$, where the covariance matrix is given by
\bal
\bm{\Sigma} = \left( (d)_{(i+j)} - (d)_{(i)}(d)_{(j)} \right)_{1 \leq i,j \leq m}. 
\eal
\end{lemma}

\begin{proof}
Let $\bm{t} = (t_1, \ldots, t_m)$. Then 
\bal
&\mathbb{E}\left( \exp\left( i \bm{t} \cdot \left( \frac{1}{\sqrt{n}} \sum_{k=1}^n (Y_k - \mathbb{E}(Y_1)), \ldots, \frac{1}{\sqrt{n}} \sum_{k=1}^n (Y_k^m - \mathbb{E}(Y_1^m)) \right) \right)\right) \\
&=  \mathbb{E}\left[ \exp\left( \frac{i}{\sqrt{n}} \sum_{k=1}^n \left[ t_1(Y_k - \mathbb{E}(Y_1)) + t_2 (Y_k^2 - \mathbb{E}(Y_1^2)) + \dotsb + t_m (Y_k^m - \mathbb{E}(Y_1^m)) \right] \right) \right] \\
&= \left(\mathbb{E}\left[ \exp\left( \frac{i}{\sqrt{n}}\left[ t_1(Y_k - \mathbb{E}(Y_1)) + t_2 (Y_k^2 - \mathbb{E}(Y_1^2)) + \dotsb + t_m (Y_k^m - \mathbb{E}(Y_1^m)) \right] \right) \right] \right)^n,
\eal
where the last equality follows by the independence of the $Y_k$'s. 

The inner expectation is computed to be
\bal
&\mathbb{E}\left(\exp\left( \frac{i}{\sqrt{n}}\left[ t_1(Y_k - \mathbb{E}(Y_1)) + t_2 (Y_k^2 - \mathbb{E}(Y_1^2)) + \dotsb + t_m (Y_k^m - \mathbb{E}(Y_1^m)) \right] \right) \right) \\ 
&= \sum_{\ell \geq 0} \frac{\left(\frac{i}{\sqrt{n}} \right)^\ell}{\ell!} \mathbb{E}\left( \left[t_1(Y_k - \mathbb{E}(Y_1)) + t_2 (Y_k^2 - \mathbb{E}(Y_1^2)) + \dotsb + t_m (Y_k^m - \mathbb{E}(Y_1^m))\right]^\ell \right) \\
&= 1 - \frac{1}{2n} \mathbb{E}\left( \left[t_1(Y_k - \mathbb{E}(Y_1)) + t_2 (Y_k^2 - \mathbb{E}(Y_1^2)) + \dotsb + t_m (Y_k^m - \mathbb{E}(Y_1^m))\right]^2 \right) + R(t),
\eal
where the remainder term $R(t)$ goes to $0$ as $\alpha, n \to \infty$. 

The expected value of the $\ell = 2$ term can be written as
\bal
&\mathbb{E}\left( \left[t_1(Y_k - \mathbb{E}(Y_1)) + t_2 (Y_k^2 - \mathbb{E}(Y_1^2)) + \dotsb + t_m (Y_k^m - \mathbb{E}(Y_1^m))\right]^2 \right) \\
&= \sum_{1 \leq i,j \leq m} t_i t_j \mathbb{E}[(Y_k^i - \mathbb{E}(Y_1^i))(Y_k^j - \mathbb{E}(Y_1^j))] \\
&= \bm{t} \cdot \bm{\Sigma} \cdot \bm{t}^T,
\eal
where the covariance matrix is given by
\bal
\bm{\Sigma_n} = \left( \bm{\Sigma_{ij}} \right)_{1 \leq i,j \leq m} &= \left( \mathbb{E}[(Y_k^i - \mathbb{E}(Y_1^i))(Y_k^j - \mathbb{E}(Y_1^j))] \right)_{1 \leq i,j \leq m} \\
&= \left( \left(\frac{\alpha}{n}\right)_{(i+j)} - \left(\frac{\alpha}{n}\right)_{(i)}\left(\frac{\alpha}{n}\right)_{(j)} \right)_{1 \leq i,j \leq m}. 
\eal

Therefore 
\bal
&\left(\mathbb{E}\left[ \exp\left( \frac{i}{\sqrt{n}}\left[ t_1(Y_k - \mathbb{E}(Y_1)) + t_2 (Y_k^2 - \mathbb{E}(Y_1^2)) + \dotsb + t_m (Y_k^m - \mathbb{E}(Y_1^m)) \right] \right) \right] \right)^n \\
&= \left(1 - \frac{1}{2n} \bm{t} \cdot \bm{\Sigma_n} \cdot \bm{t}^T + R(t)\right)^n \\
&\to \exp\left( -\frac{1}{2} \bm{t} \cdot \bm{\Sigma} \cdot \bm{t}^T \right)
\eal
as $\alpha,n \to \infty$, which is the characteristic function of a multivariate normal with mean $0$ and covariance matrix 
\bal
\bm{\Sigma} &= \left( (d)_{(i+j)} - (d)_{(i)}(d)_{(j)} \right)_{1 \leq i,j \leq m}. \qedhere
\eal
\end{proof}

Combining the previous lemma with the multivariate delta method, we obtain the following multivariate central limit theorem. 

\begin{lemma}\label{jointdirichletclt}
Let $(W_{n1},\ldots, W_{nn}) \sim \Dir\left( \frac{\alpha}{n}, \ldots, \frac{\alpha}{n} \right)$. Then
\bal
\left( n^{3/2} \sum_{k=1}^n (W_{nk}^2 - \mathbb{E}(W_{nk}^2)),\ldots, n^{m-\frac{1}{2}} \sum_{k=1}^n (W_{nk}^m - \mathbb{E}(W_{nk}^m)) \right) \xrightarrow{D} N(\bm{0}, \bm{\Sigma^*})
\eal
as $\alpha, n \to \infty$ such that $\frac{\alpha}{n} \to d$, where the covariance matrix $\bm{\Sigma^*} = (\bm{\Sigma_{ij}^*})_{1 \leq i,j \leq m-1}$ is given by
\bal
\bm{\Sigma_{ij}^*} &= \frac{(d+1)_{(i+j+1)} - (d+1)_{(i)}(d+1)_{(j)}[(i+1)(j+1) - d]}{d^{i+j-1}}. 
\eal
\end{lemma}

\begin{proof}
Define the vector function $g:\R^m \to \R^{m-1}$ by
\bal
g(x_1, \ldots, x_m) = (g_1(x_1, \ldots, x_m), \ldots, g_{m-1}(x_1, \ldots, x_m)) = \left( \frac{n^2 x_2}{x_1^2}, \frac{n^3 x_3}{x_1^3}, \ldots, \frac{n^m x_m}{x_1^m} \right).
\eal
The gradient of $g$ is the $m \times (m-1)$ matrix given by
\bal
\nabla g(x_1, \ldots, x_m) = \begin{pmatrix} -\frac{2n^2 x_2}{x_1^3} & -\frac{3n^3 x_3}{x_1^4} & \hdots & -\frac{m n^m x_m}{x_1^{m+1}}  \\
\frac{n^2}{x_1^2} & 0 & \hdots & 0 \\
0 & \frac{n^3}{x_1^3} & \hdots & 0 \\
\vdots & \vdots & \ddots & \vdots \\
0 & 0 & \hdots & \frac{n^m}{x_1^m} \\
\end{pmatrix}
\eal
By Lemma \ref{jointgammaclt},
\bal
\frac{1}{\sqrt{n}} \left[ \begin{pmatrix} \sum_{k=1}^n Y_k \\ \sum_{k=1}^n Y_k^2 \\ \vdots \\ \sum_{k=1}^n Y_k^m \end{pmatrix} - \begin{pmatrix} \mathbb{E}\left(\sum_{k=1}^n Y_k\right) \\ \mathbb{E}\left(\sum_{k=1}^n Y_k^2 \right) \\ \vdots \\ \mathbb{E}\left(\sum_{k=1}^n Y_k^m \right) \end{pmatrix} \right] \xrightarrow{D} N(\bm{0}, \bm{\Sigma}),
\eal
where $\bm{\Sigma} = \left( (d)_{(i+j)} - (d)_{(i)}(d)_{(j)} \right)_{1 \leq i,j \leq m}$. 
We can write
\bal
&\begin{pmatrix} n^{3/2} \sum_{k=1}^n (W_{nk}^2 - \mathbb{E}(W_{nk}^2)) \\ n^{5/2} \sum_{k=1}^n (W_{nk}^3 - \mathbb{E}(W_{nk}^3)) \\ \vdots \\ n^{m - 1/2} \sum_{k=1}^n (W_{nk}^m - \mathbb{E}(W_{nk}^m)) \end{pmatrix} \\
&\qquad = \frac{1}{\sqrt{n}} \left[ g\begin{pmatrix} \sum_{k=1}^n Y_k \\ \sum_{k=1}^n Y_k^2 \\ \vdots \\ \sum_{k=1}^n Y_k^m \end{pmatrix} - g \begin{pmatrix} \mathbb{E}\left(\sum_{k=1}^n Y_k\right) \\ \mathbb{E}\left(\sum_{k=1}^n Y_k^2 \right)\\ \vdots \\ \mathbb{E}\left(\sum_{k=1}^n Y_k^m \right) \end{pmatrix} \right] 
+ O\begin{pmatrix} 1/\sqrt{n} \\ \vdots \\ 1/\sqrt{n} \end{pmatrix},
\eal 
where the vector of constant terms converges to $\bm{0}$ as $n \to \infty$. 

Therefore by the multivariate delta method, 
\bal
\begin{pmatrix} n^{3/2} \sum_{k=1}^n (W_{nk}^2 - \mathbb{E}(W_{nk}^2)) \\ n^{5/2} \sum_{k=1}^n (W_{nk}^3 - \mathbb{E}(W_{nk}^3)) \\ \vdots \\ n^{m - 1/2} \sum_{k=1}^n (W_{nk}^m - \mathbb{E}(W_{nk}^m)) \end{pmatrix} \xrightarrow{D} N(\bm{0}, \bm{\Sigma^*}), 
\eal 
where the $(m-1) \times (m-1)$ covariance matrix is given by $\bm{\Sigma^*} = \lim_{n \to \infty} \bm{\Sigma_n^*}$ and 
\bal
\bm{\Sigma_n^*} &= \left[\nabla g \left( \mathbb{E}\left(\sum_{k=1}^n Y_k\right), \mathbb{E}\left(\sum_{k=1}^n Y_k^2 \right), \ldots, \mathbb{E}\left(\sum_{k=1}^n Y_k^m \right) \right) \right]^T \\
&\qquad \cdot \bm{\Sigma} \cdot \nabla g \left( \mathbb{E}\left(\sum_{k=1}^n Y_k\right), \mathbb{E}\left(\sum_{k=1}^n Y_k^2 \right), \ldots, \mathbb{E}\left(\sum_{k=1}^n Y_k^m \right) \right). 
\eal
Noting that $\mathbb{E}(\sum_{k=1}^n Y_k^j) = n\left( \frac{\alpha}{n} \right)_{(j)} \to n(d)_{(j)}$, we get
\bal
&\nabla g \left( \mathbb{E}\left(\sum_{k=1}^n Y_k\right), \mathbb{E}\left(\sum_{k=1}^n Y_k^2 \right), \ldots, \mathbb{E}\left(\sum_{k=1}^n Y_k^m \right) \right) \\
&\qquad = \begin{pmatrix} -\frac{2\left( \frac{\alpha}{n} + 1 \right)_{(1)}}{\left( \frac{\alpha}{n} \right)^2} & -\frac{3\left( \frac{\alpha}{n} + 1\right)_{(2)}}{\left( \frac{\alpha}{n} \right)^3} & \hdots & -\frac{m\left( \frac{\alpha}{n} + 1 \right)_{(m-1)}}{\left( \frac{\alpha}{n} \right)^m}  \\
\frac{1}{\left( \frac{\alpha}{n} \right)^2} & 0 & \hdots & 0 \\
0 & \frac{1}{\left( \frac{\alpha}{n} \right)^3} & \hdots & 0 \\
\vdots & \vdots & \ddots & \vdots \\
0 & 0 & \hdots & \frac{1}{\left( \frac{\alpha}{n} \right)^m} \\
\end{pmatrix} \\
&\qquad \to \begin{pmatrix} -\frac{2(d+1)_{(1)}}{d^2} & -\frac{3(d+1)_{(2)}}{d^3} & \hdots & -\frac{m(d+1)_{(m-1)}}{d^m}  \\
\frac{1}{d^2} & 0 & \hdots & 0 \\
0 & \frac{1}{d^3} & \hdots & 0 \\
\vdots & \vdots & \ddots & \vdots \\
0 & 0 & \hdots & \frac{1}{d^m} \\
\end{pmatrix}
\eal
as $\alpha, n \to \infty$, so that
\bal
\bm{\Sigma^*} &= \begin{pmatrix} -\frac{2(d+1)_{(1)}}{d^2} & -\frac{3(d+1)_{(2)}}{d^3} & \hdots & -\frac{m(d+1)_{(m-1)}}{d^m}  \\
\frac{1}{d^2} & 0 & \hdots & 0 \\
0 & \frac{1}{d^3} & \hdots & 0 \\
\vdots & \vdots & \ddots & \vdots \\
0 & 0 & \hdots & \frac{1}{d^m} \\
\end{pmatrix}^T \\
&\qquad \cdot \bm{\Sigma} \cdot \begin{pmatrix} -\frac{2(d+1)_{(1)}}{d^2} & -\frac{3(d+1)_{(2)}}{d^3} & \hdots & -\frac{m(d+1)_{(m-1)}}{d^m}  \\
\frac{1}{d^2} & 0 & \hdots & 0 \\
0 & \frac{1}{d^3} & \hdots & 0 \\
\vdots & \vdots & \ddots & \vdots \\
0 & 0 & \hdots & \frac{1}{d^m} \\
\end{pmatrix}.
\eal
By direct computation, the $(i,j)$th entry is 
\bal
\bm{\Sigma_{ij}^*} &= \frac{(d+1)_{(i+j+1)} - (d+1)_{(i)}(d+1)_{(j)}[(i+1)(j+1) + d]}{d^{i+j-1}}. \qedhere
\eal
\end{proof}

By Lemmas \ref{jointgammaclt} and \ref{jointdirichletclt}, we obtain the central limit theorem for the level 1 contribution. 

\begin{lemma} \label{Level1CLTFDHDP}
The level 1 contribution satisfies
\bal
Y_m(\alpha,\beta,n) &= \frac{\sqrt{\beta}}{f(\beta; m,c,d)} \left( \sum_{k=1}^n \frac{\Gamma(\beta W_{nk} + m)}{\beta^m \Gamma(\beta W_{nk})} - \frac{1}{\beta^{m-1}} \sum_{j=1}^m {m \brack j} \frac{\beta^{j-1} \left( \frac{\alpha}{n} + 1 \right)_{(j-1)}}{(\alpha+1)_{(j-1)}} \right) \\
&\xrightarrow{D} N(0,v^2_1(c,d))
\eal
as $\alpha, \beta, n \to \infty$ such that $\frac{\alpha}{\beta} \to c$ and $\frac{\alpha}{n} \to d$, where the variance is given by
\bal
&v^2_1(c,d) = \frac{1}{\left( \sum_{j=1}^m {m \brack j} \frac{(d+1)_{(j-1)}}{c^{j-1}} \right)^2} \\
&\qquad \times \sum_{1 \leq i,j \leq m} {m \brack i} {m \brack j} \frac{(d+1)_{(i+j-1)} - (d+1)_{(i-1)}(d+1)_{(j-1)}(ij+d)}{c^{i+j-1}}.
\eal
\end{lemma}

\subsection{Joint Convergence in the FDHDP}

Let 
\bal
Z_m(\alpha, \beta, n) = \sqrt{\beta} \left( \sum_{k=1}^n \frac{\gamma_k(\alpha, \beta, n)}{\beta} - 1  \right) = \sqrt{\beta} \sum_{k=1}^n \left( \frac{\gamma_k(\alpha, \beta, n)}{\beta} - W_{nk}  \right).
\eal
We will need the following joint convergence result. 

\begin{lemma} \label{JointConvFDHDP}
We have the following multivariate central limit theorem
\bal
(Z_m(\alpha, \beta, n), X_m(\alpha, \beta, n)) \xrightarrow{D} N(\bm{0}, \bm{\Sigma})
\eal
as $\alpha, \beta, n \to \infty$ such that $\frac{\alpha}{\beta} \to c$ and $\frac{\alpha}{n} \to d$, where the covariance matrix is given by
\bal
\bm{\Sigma} = \begin{pmatrix}
1 & m \\
m & v^2_2(c,d)
\end{pmatrix},
\eal
where $v_2^2(c,d)$ is the level 2 variance from Lemma \ref{Level2CLTFDHDP}. 
\end{lemma}

The proof uses characteristic functions and is similar to the joint convergence result for the HDP case.

\subsection{Proof of the CLT for the Homozygosity of the FDHDP}

We are now in a position to prove the central limit theorem for the homozygosity of the FDHDP. 

\begin{proof}[Proof of Theorem \ref{CLTHomozygosityFDHDP}]
Combining Lemmas \ref{Level2CLTFDHDP}, \ref{Level1CLTFDHDP}, \ref{JointConvFDHDP}, \ref{Fen7.7}, Slutsky's theorem, 
the fact that $(X_m(\alpha, \beta, n), Y_m(\alpha, \beta, n)) \xrightarrow{D} N\left( \bm{0}, \begin{pmatrix} v_2^2(c,d) & 0 \\ 0 & v_1^2(c,d) \end{pmatrix} \right)$, and the fact that $Y_m(\alpha, \beta, n)$ is independent of $Z_m(\alpha,\beta,n)$, we have that for any $m \geq 2$,
\bal
\tilde{H}_{m,n}(\alpha, \beta) \xrightarrow{D} X + Y - mZ,
\eal
as $\alpha, \beta, n \to \infty$ such that $\frac{\alpha}{\beta} \to c$ and $\frac{\alpha}{n} \to d$, 
which is distributed as a normal random variable $N(0, v^2(c,d))$, with variance given by
\bal
v^2(c,d) &= \Var(X) + \Var(Y) + m^2 \Var(Z) - 2m\cov(X,Z) \\
&= v_1^2(c,d) + v_2^2(c,d) - m^2 \\
&= \frac{\sum_{j=1}^{2m} {2m \brack j} \frac{(d+1)_{(j-1)}}{c^{j-1}} - \sum_{1 \leq i,j \leq m} {m \brack i}{m \brack j} \frac{(d+1)_{(i-1)}(d+1)_{(j-1)}(ij+d)}{c^{i+j-1}}}{\left( \sum_{j=1}^m {m \brack j} \frac{(d+1)_{(j-1)}}{c^{j-1}} \right)^2} - m^2. \qedhere
\eal 
\end{proof}

\section{CLT for the Homozygosity of the HDP with $L$ Groups} \label{CLTHomozygHDPgroups}

Recall that the homozygosity of the HDP with $L$ groups is defined as
\bal
H^L_m(\alpha,\beta) &= \sum_{i=1}^\infty \sum_{\bm{m} \in M_{m,L}} \frac{1}{L^m} \prod_{k=1}^L Z_{k,i}^{m_k}
\eal
In this section, we prove the central limit theorem for $H^L_{m}(\alpha, \beta)$. The additional challenge in the multigroup setting is to understand the interactions between the different groups. 

\subsection{Gamma Representation, Law of Large Numbers, and Decomposition for the HDP with $L$ Groups}

The HDP weights of each group has subordinator representation
\bal
\bm{Z}_k = (Z_{k,1}, Z_{k,2}, \ldots) = \left( \frac{\gamma_{k1(\alpha, \beta)}}{\gamma_k(\beta)}, \frac{\gamma_{k2(\alpha, \beta)}}{\gamma_k(\beta)}, \ldots \right)
\eal
where $\gamma_{ki}(\alpha, \beta) \sim \Gam(\beta V_i, 1)$ independently of $\gamma_k(\beta) \sim \Gam(\beta, 1)$, for all $1 \leq k \leq L$. We start by computing the expected value of $H^L_{m}(\alpha, \beta)$. 

\begin{lemma} \label{homozygositygroupsmean}
For all $m \geq 2$, 
\bal
\E[H^L_{m}(\alpha, \beta)] &= \frac{1}{L^m} \sum_{\bm{m} \in M_{m,L}} \frac{1}{\prod_{k=1}^L (\beta)_{(m_k)}} \sum_{j = 1}^{m} a_j(\bm{m},L) \beta^j \frac{\Gamma(j)}{(\alpha + 1)_{(j-1)}}, 
\eal
where the coefficients $\{a_j(\bm{m},L)\}_{1 \leq j \leq m}$ are given by 
\bal
a_j(\bm{m},L) = \sum_{\bm{j} \in M_{j,L}} \prod_{k=1}^L {m_k \brack j_k}
\eal 
for all $\bm{m} \in M_{m,L}$. 
\end{lemma}

\noindent{\bf Remark:}  Consider the case $m=2$ and $L = 2$. Given that two samples are from one specific group, the within-group co-clustering probability that the samples have the same type is 
\bal
\frac{1}{\beta(\beta+1)}\sum_{j=1}^2 a_j((2,0), 2)\beta^j \frac{\Gamma(j)}{(\alpha+1)_{(j-1)}}&= \frac{\alpha+\beta+1}{(\alpha+1)(\beta+1)}. 
\eal 
Similarly the across-group co-clustering probability that two samples from two different groups have the same type is 
\bal
\frac{1}{\beta^2}\sum_{j=1}^2 a_j((1,1), 2)\beta^j \frac{\Gamma(j)}{(\alpha+1)_{(j-1)}}&= \frac{1}{\alpha+1}. 
\eal 
These are precisely the respective within-group and across-group correlations observed by Ascolani et. al. in \cite{AFLP24}, Example 1.

\begin{proof}
Conditioning on $\bm{V}$, using Lemma \ref{Zmoments}, and some algebra yields
\bal
\E[H^L_{m}(\alpha, \beta) \mid \bm{V}] 
&= \frac{1}{L^m} \sum_{i=1}^\infty \sum_{\bm{m} \in M_{m,L}} \prod_{k=1}^L \left( \frac{1}{(\beta)_{(m_k)} }\sum_{j=1 \wedge m_k}^{m_k} {m_k \brack j} \beta^j V_i^j \right) \\
&= \frac{1}{L^m \prod_{k=1}^L (\beta)_{(m_k)}} \sum_{j = 1}^{m} \left(\sum_{\bm{j} \in M_{j,L}} \prod_{k=1}^L {m_k \brack j_k} \right)  \beta^j \sum_{i=1}^\infty V_i^j \\ &= \frac{1}{L^m \prod_{k=1}^L (\beta)_{(m_k)}} \sum_{j = 1}^{m} a_j(\bm{m},L) \beta^j H_j(\alpha),
\eal
where the coefficients $\{a_j(\bm{m},L)\}_{1 \leq j \leq m}$ are defined as in the lemma statement. 
Thus 
\bal
\E[H^L_{m}(\alpha, \beta)] &= \frac{1}{L^m} \sum_{\bm{m} \in M_{m,L}} \frac{1}{\prod_{k=1}^L (\beta)_{(m_k)}} \sum_{j = 1}^{m} a_j(\bm{m},L) \beta^j \frac{\Gamma(j)}{(\alpha + 1)_{(j-1)}}. \qedhere
\eal
\end{proof} 

Observe that the mean of the homozygosity is asymptotically
\bal
\E[H^L_{m}(\alpha, \beta)] &\asymp \frac{1}{L^m \beta^{m-1}} \sum_{j = 1}^{m} A_j(m,L) \frac{\Gamma(j)}{c^{j-1}} = f(\beta; m,L,c),
\eal
as $\alpha, \beta \to \infty$ such that $\frac{\alpha}{\beta} \to c$, where the coefficients $\{A_j(m,L)\}_{1 \leq j \leq m}$ are defined by
\bal
A_j(m,L) = \sum_{\bm{m} \in M_{m,L}} a_j(\bm{m},L). 
\eal
The next result establishes the law of large numbers for $H^L_{m}(\alpha, \beta)$. The proof follows by a Chebyshev's inequality argument as in the $L = 1$ case, Proposition \ref{LLN}. 

\begin{proposition} \label{LLNDgroups}
The homozygosity satisfies
\bal
\frac{H^L_{m}(\alpha, \beta)}{f(\beta;m,L,c)} \to 1
\eal
in probability as $\alpha, \beta \to \infty$ such that $\frac{\alpha}{\beta} \to c$. 
\end{proposition}

The law of large numbers suggests a scaling of the form
\bal
\tilde{H}^L_{m}(\alpha, \beta) = \frac{\sqrt{\beta}}{f(\beta;m,L,c)}\left( H^L_{m}(\alpha, \beta) - \frac{1}{L^m \beta^{m-1}}  \sum_{j = 1}^{m} A_j(m,L) \frac{\Gamma(j)}{(\alpha/\beta)^{j-1}} \right), 
\eal
which we call the {\em scaled $m$th order homozygosity}. The homozygosity can be decomposed in a similar way as the $L = 1$ case. First note that
\bal
\E\left[ \sum_{m \in M_{m,L}} \prod_{k=1}^L \gamma_{ki}^{m_k}(\alpha, \beta) \middle \vert \bm{V} \right] = \sum_{j=1}^m A_j(m,L) \beta^j V_i^j, 
\eal
and define the level 1 and level 2 contributions, respectively, as 
\bal
X_{m,L}(\alpha, \beta) &= \frac{\sqrt{\beta}}{f(\beta;m,L,c)} \sum_{i=1}^\infty \frac{\sum_{\bm{m} \in M_{m,L}}  \prod_{k=1}^L \gamma_{ki}^{m_k}(\alpha, \beta) - \sum_{j=1}^m A_j(m,L) \beta^j V_i^j }{L^m \beta^m} \\
Y_{m,L}(\alpha, \beta) &= \frac{\sqrt{\beta}}{f(\beta;m,L,c)} \left( \sum_{i=1}^\infty \frac{\sum_{j=1}^m A_j(m,L) \beta^j V_i^j}{L^m \beta^m} - \frac{1}{L^m \beta^{m-1}}  \sum_{j = 1}^{m} A_j(m,L) \frac{\Gamma(j)}{(\alpha/\beta)^{j-1}} \right).
\eal
Then the scaled homozygosity can be decomposed as 
\bal
\tilde{H}^L_{m}(\alpha, \beta) &= X_{m,L}(\alpha, \beta) + Y_{m,L}(\alpha, \beta) \\
&\qquad + \sum_{\bm{m} \in M_{m,L}} \sqrt{\beta} \left( \prod_{k=1}^L \left(\frac{\beta}{\gamma_{k}(\beta)}\right)^{m_k} - 1 \right) \frac{ \frac{1}{L^m} \sum_{i=1}^\infty \prod_{k=1}^L \left( \frac{\gamma_{ki}(\alpha, \beta)}{\beta} \right)^{m_k} }{f(\beta;m,L,c)}.
\eal

\subsection{Level 1 Contribution for the HDP with $L$ Groups}

In this section, we prove the central limit theorem for $Y_{m,L}(\alpha, \beta)$. The proof is similar to the $L = 1$ case. 

\begin{lemma} \label{level1CLTgroups}
The level 1 contribution satisfies
\bal
Y_{m,L}(\alpha, \beta) \xrightarrow{D} N(0,\sigma_{1,L}^2),
\eal
where 
\bal
\sigma_{1,L}^2 = \frac{\sum_{1 \leq i,j \leq m} A_i(m,L)A_j(m,L) \frac{\Gamma(i+j) - \Gamma(i+1)\Gamma(j+1)}{c^{i+j-1}}}{\left( \sum_{j = 1}^{m} A_j(m,L) \frac{\Gamma(j)}{c^{j-1}}\right)^2}.
\eal
\end{lemma}

\begin{proof}
By direct computation, we can rewrite
\bal
&Y_{m,L}(\alpha, \beta) = \frac{\sqrt{\frac{\beta}{\alpha}} \sum_{j=2}^m A_j(m,L) \frac{\Gamma(j)}{(\alpha/\beta)^{j-1}} \tilde{H}_j(\alpha) }{\sum_{j = 1}^{m} A_j(m,L) \frac{\Gamma(j)}{c^{j-1}}},
\eal
where $\tilde{H}_j(\alpha) = \sqrt{\alpha}\left( \frac{\alpha^{j-1}}{\Gamma(j)} \sum_{k=1}^\infty V_k^j - 1 \right)$ is the scaled level 1 homozygosity. From this representation it follows that $Y_{m,D}$ is asymptotically normal with mean $0$ and variance
\bal
\sigma_{1,L}^2 &= \frac{\sum_{1 \leq i,j \leq m} A_i(m,L)A_j(m,L) \frac{\Gamma(i+j) - \Gamma(i+1)\Gamma(j+1)}{c^{i+j-1}}}{\left( \sum_{j = 1}^{m} A_j(m,L) \frac{\Gamma(j)}{c^{j-1}}\right)^2}. \qedhere
\eal
\end{proof}

\subsection{Level 2 Contribution for the HDP with $L$ Groups}

In this section, we prove the central limit theorem for $X_{m,L}(\alpha, \beta)$. The proof follows by the same characteristic function approach as in the $L = 1$ case, and so we omit the details but highlight the new aspects. 

\begin{lemma} \label{level2CLTgroups}
The level 2 contribution satisfies
\bal
X_{m,L}(\alpha, \beta) \xrightarrow{D} N(0,\sigma_{2,L}^2),
\eal
where 
\bal
\sigma_{2, L}^2 &= \frac{\sum_{j=1}^{2m} \tilde{A}_j(m,L) \frac{\Gamma(j)}{c^{j-1}} - \sum_{1 \leq i,j \leq m} A_i(m,L)A_j(m,L) \frac{\Gamma(i+j)}{c^{i+j-1}}}{\left( \sum_{j = 1}^{m} A_j(m,L) \frac{\Gamma(j)}{c^{j-1}}\right)^2}
\eal
and the coefficients $\{\tilde{A}_j(m,L)\}_{1 \leq j \leq 2m}$ are defined by
\bal
\tilde{A}_j(m,L) &= \sum_{\bm{m}_1 \in M_{m,L}} \sum_{\bm{m}_2 \in M_{m,L}} \left( \sum_{\bm{j} \in M_{j,L}} \prod_{k=1}^L {m_{1k} + m_{2k} \brack j_k} \right).
\eal
\end{lemma}

\begin{proof}
It suffices to show that the characteristic function of the truncated sum 
\bal
X_{m, L, 1}(\alpha, \beta) &= \frac{\sqrt{\beta}}{f(\beta;m,L,c)} \sum_{r = 1}^{\lfloor \alpha^2 \rfloor} \frac{\sum_{\bm{m} \in M_{m,L}}  \prod_{k=1}^L \gamma_{kr}^{m_k}(\alpha, \beta) - \sum_{j=1}^m A_j(m,L) \beta^j V_r^j }{L^m \beta^m}
\eal
converges to the characteristic function of $N(0,\sigma_{2,L}^2)$. Conditioning on $\bm{V}$, we get
\bal
&\E\left( \exp\left( it \frac{\sqrt{\beta}}{f(\beta;m,L,c)} \sum_{r=1}^{\lfloor \alpha^2 \rfloor} \frac{\sum_{\bm{m} \in M_{m,L}}  \prod_{k=1}^L \gamma_{kr}^{m_k}(\alpha, \beta) - \sum_{j=1}^m A_j(m,L) \beta^j V_r^j }{L^m \beta^m} \right) \right) \\
&\qquad \qquad = \E\left[ \prod_{r=1}^{\lfloor \alpha^2 \rfloor} \left( 1 + c_{\alpha, \beta, r} \right) \right],
\eal
such that
\bal
c_{\alpha, \beta, r} = -\frac{t^2}{2} \left[ \frac{\sum_{j=1}^{2m} \tilde{A}_j(m,L) \beta^{j-1} V_r^j - \sum_{1 \leq i,j \leq m} A_i(m,L) A_j(m,L)\beta^{i+j-1} V_r^{i+j}}{ \left( \sum_{j = 1}^{m} A_j(m,L) \frac{\Gamma(j)}{c^{j-1}}\right)^2} \right] + O\left( \frac{1}{\sqrt{\beta}} \right),
\eal
where we computed the second moment
\bal
&\E\left[ \left( \sum_{\bm{m} \in M_{m,L}}  \prod_{k=1}^L \gamma_{kr}^{m_k}(\alpha, \beta) \right)^2 \middle\vert \bm{V} \right] = \sum_{\bm{m}_1 \in M_{m,L}} \sum_{\bm{m}_2 \in M_{m,L}} \prod_{k=1}^L \E\left[ \gamma_{kr}^{m_{1k} + m_{2k}}(\alpha, \beta) \middle\vert \bm{V} \right] \\
& = \sum_{\bm{m}_1 \in M_{m,L}} \sum_{\bm{m}_2 \in M_{m,L}} \sum_{j=1}^{2m} \left( \sum_{\bm{j} \in M_{j,L}} \prod_{k=1}^L {m_{1k} + m_{2k} \brack j} \right) \beta^j V_r^j = \sum_{j=1}^{2m} \tilde{A}_j(m,D) \beta^j V_r^j
\eal
with the coefficients $\{\tilde{A}_j(m,L)\}_{1 \leq j \leq 2m}$ defined as in the lemma statement. Applying Lemma \ref{as-LLN} allows us to identify the asymptotic variance 
\bal
\lim_{\alpha, \beta \to \infty, \alpha/\beta \to c} \sum_{r=1}^{\lfloor \alpha^2 \rfloor} c_{\alpha, \beta, r} = -\frac{t^2}{2}\left( \frac{\sum_{j=1}^{2m} \tilde{A}_j(m,L) \frac{\Gamma(j)}{c^{j-1}} - \sum_{1 \leq i,j \leq m} A_i(m,L) A_j(m,L) \frac{\Gamma(i+j)}{c^{i+j-1}}}{ \left( \sum_{j = 1}^{m} A_j(m,L) \frac{\Gamma(j)}{c^{j-1}}\right)^2} \right).
\eal
Finally by dominated convergence, the unconditional characteristic function of $X_{m, L, 1}(\alpha, \beta)$ also converges to the characteristic function of $N(0,\sigma_{2,L}^2)$, and the result follows. 
\end{proof}

\subsection{Joint Convergence in the HDP with $L$ Groups}

For any $1\leq k\leq L$, let 
\bal
Z_k(\alpha,\beta) = \sqrt{\beta}\left( \sum_{i=1}^\infty \frac{\gamma_{ki}(\alpha, \beta)}{\beta} - 1 \right)
\eal
for all $1 \leq k \leq L$. We will need the following joint convergence result. The proof uses characteristic functions and is similar to the $L = 1$ case, and so we only highlight the differences. 

\begin{lemma} \label{JointCLTDgroups}
For all $1 \leq k\leq L$, we have the following multivariate central limit theorem 
\bal
(Z_k(\alpha, \beta), X_{m,L}(\alpha, \beta)) \xrightarrow{D} N(\bm{0}, \bm{\Sigma^{(k)}})
\eal
as $\alpha, \beta \to \infty$ such that $\frac{\alpha}{\beta} \to c$, where the covariance matrix is given by
\bal
\bm{\Sigma^{(k)}} = \begin{pmatrix}
1 & \sum_{\bm{m} \in M_{m,L}} m_k C_{\bm{m}} \\ \sum_{\bm{m} \in M_{m,L}} m_k C_{\bm{m}} & \sigma_{2,L}^2
\end{pmatrix}, 
\eal
$\sigma_{2,L}^2$ is the level 2 variance from Lemma \ref{level2CLTgroups}, and the constants $\{C_{\bm{m}}\}$ are defined by
\bal
C_{\bm{m}} = \frac{\sum_{j=1}^m a_j(\bm{m},L) \frac{\Gamma(j)}{c^{j-1}}}{\sum_{j=1}^m A_j(m,L) \frac{\Gamma(j)}{c^{j-1}}}.
\eal
\end{lemma}

\begin{proof}
The only difference from the proof of the $L = 1$ case is in computation of the third (cross) summand, which is computed to be
\bal
&\E\left[ \left( \frac{\gamma_{ki}(\alpha, \beta) - \E[\gamma_{ki}(\alpha, \beta) \mid \bm{V}]}{\beta} \right) \right. \\
&\qquad \qquad \left. \times \left( \frac{ \sum_{\bm{m} \in M_{m,L}} \prod_{j=1}^L \gamma_{ji}^{m_j}(\alpha, \beta) -  \E\left[ \sum_{\bm{m} \in M_{m,L}} \prod_{j=1}^L \gamma_{ji}^{m_j}(\alpha, \beta) \mid \bm{V} \right]}{L^m \beta^m} \right) \right] \\
&\quad = \frac{1}{L^m \beta^{m+1}}\sum_{\bm{m} \in M_{m,L}} \left[ \left(\prod_{j \neq k} \frac{\Gamma(\beta V_i + m_j)}{\Gamma(\beta V_i)}\right) \left( \frac{\Gamma(\beta V_i + m_k+ 1) - \beta V_i \Gamma(\beta V_i + m_k)}{\Gamma(\beta V_i)}\right) \right] \\
&\quad = \frac{1}{L^m \beta^{m+1}}\sum_{\bm{m} \in M_{m,L}} m_k \sum_{j=1}^m a_j(\bm{m},L) \beta^j V_i^j.
\eal
Therefore the asymptotic covariance between $X_{m,L}(\alpha, \beta)$ and $Z_k(\alpha, \beta)$ is
\bal
&\frac{\beta}{f(\beta;m,L,c)} \sum_{i=1}^\infty \frac{1}{L^m \beta^{m+1}}\sum_{\bm{m} \in M_{m,L}} m_k\sum_{j=1}^m a_j(\bm{m},L) \beta^j V_i^j \\
&\qquad \qquad = \frac{\sum_{\bm{m} \in M_{m,L}} m_k \sum_{j=1}^m a_j(\bm{m},L) \frac{\Gamma(j)}{(\alpha/\beta)^{j-1}}}{\sum_{j=1}^m A_j(m,L) \frac{\Gamma(j)}{c^{j-1}}} \to \sum_{\bm{m} \in M_{m,L}} m_k C_{\bm{m}},
\eal
as $\alpha, \beta \to \infty$ such that $\frac{\alpha}{\beta} \to c$. 
\end{proof}

\subsection{Proof of the CLT for the Homozygosity of the HDP with $L$ Groups}

We are now in a position to prove the central limit theorem for $H^L_{m}(\alpha, \beta)$. 

\begin{proof}[Proof of Theorem \ref{CLTHomozygosityHDPgroups}]
Recall that for $m \geq 2$, 
\bal
\tilde{H}^L_{m}(\alpha, \beta) &= X_{m,L}(\alpha, \beta) + Y_{m,L}(\alpha, \beta) \\
&\qquad + \sum_{\bm{m} \in M_{m,L}} \sqrt{\beta} \left( \prod_{k=1}^L \left(\frac{\beta}{\gamma_{k}(\beta)}\right)^{m_k} - 1 \right) \frac{ \frac{1}{L^m} \sum_{i=1}^\infty \prod_{k=1}^L \left( \frac{\gamma_{ki}(\alpha, \beta)}{\beta} \right)^{m_k} }{f(\beta;m,L,c)}.
\eal
By Lemmas \ref{level1CLTgroups} and \ref{level2CLTgroups},
\bal
X_{m,L}(\alpha, \beta) \xrightarrow{D} X \sim N(0, \sigma_{2,L}^2) \quad \text{and} \quad Y_{m,L}(\alpha, \beta) &\xrightarrow{D} Y \sim N(0, \sigma_{1,L}^2)
\eal
as $\alpha, \beta \to \infty$ such that $\frac{\alpha}{\beta} \to c$. By the law of large numbers,
\bal
\frac{ \frac{1}{L^m} \sum_{i=1}^\infty \prod_{k=1}^L \left( \frac{\gamma_{ki}(\alpha, \beta)}{\beta} \right)^{m_k} }{f(\beta;m,L,c)} \xrightarrow{D} \frac{ \sum_{j=1}^m a_j(\bm{m},L) \frac{\Gamma(j)}{c^{j-1}} }{ \sum_{j=1}^m A_j(m,L) \frac{\Gamma(j)}{c^{j-1}} }= C_{\bm{m}}
\eal
as $\alpha, \beta \to \infty$ such that $\frac{\alpha}{\beta} \to c$, where $C_{\bm{m}}$ are the constants defined in Lemma \ref{JointCLTDgroups}. 
Next we can rewrite 
\bal
&\sum_{\bm{m} \in M_{m,L}} \sqrt{\beta} \left( \prod_{k=1}^L \left(\frac{\beta}{\gamma_{k}(\beta)}\right)^{m_k} - 1 \right) \frac{ \frac{1}{L^m} \sum_{i=1}^\infty \prod_{k=1}^L \left( \frac{\gamma_{ki}(\alpha, \beta)}{\beta} \right)^{m_k} }{f(\beta;m,L,c)} \\
&= \sum_{\bm{m} \in M_{m,L}} C_{\bm{m}} \sqrt{\beta} \left( \prod_{k=1}^L \left(\frac{\beta}{\gamma_{k}(\beta)}\right)^{m_k} - 1 \right) \\
&\qquad + \sum_{\bm{m} \in M_{m,L}} \sqrt{\beta} \left( \prod_{k=1}^L \left(\frac{\beta}{\gamma_{k}(\beta)}\right)^{m_k} - 1 \right) \left(\frac{ \frac{1}{L^m} \sum_{i=1}^\infty \prod_{k=1}^L \left( \frac{\gamma_{ki}(\alpha, \beta)}{\beta} \right)^{m_k} }{f(\beta;m,L,c)} - C_{\bm{m}} \right).
\eal
The second summand converges in distribution to $0$. By Lemma \ref{Fen7.7}, 
\bal
Z_k(\alpha, \beta) = \sqrt{\beta}\left( \sum_{i=1}^\infty \frac{\gamma_{ki}(\alpha, \beta)}{\beta} - 1 \right) = \sqrt{\beta}\left( \frac{\gamma_{k}(\beta)}{\beta} - 1 \right) \xrightarrow{D} N(0,1)
\eal
as $\beta \to \infty$, for all $1 \leq k \leq L$. Using the fact that $(\gamma_1(\beta), \ldots, \gamma_L(\beta))$ are conditionally independent given $\bm{V}$, we have the joint convergence 
\bal
\sqrt{\beta}\left[ \begin{pmatrix} \frac{\gamma_1}{\beta} \\ \vdots \\ \frac{\gamma_L}{\beta} \end{pmatrix} - \begin{pmatrix} 1 \\ \vdots \\ 1 \end{pmatrix} \right] \xrightarrow{D} N(\bm{0}, \bm{I}_L)
\eal
as $\beta \to \infty$, where $\bm{I}_L$ is the $L \times L$ identity matrix. Define the function $h : \R^L \to \R$ by
\bal
h(x_1, \ldots, x_L) = \sum_{\bm{m} \in M_{m,L}} C_{\bm{m}} \prod_{k=1}^L x_k^{-m_k}.
\eal
By direct computation,
\bal
\nabla h(1, \ldots, 1)^T \cdot I \cdot \nabla h(1, \ldots, 1) & = \sum_{k=1}^L \left( \sum_{\bm{m} \in M_{m,L}} C_{\bm{m}} m_k \right)^2,
\eal
and so by the multivariate delta method, 
\bal
\sum_{\bm{m} \in M_{m,L}} C_{\bm{m}} \sqrt{\beta}\left( \prod_{k=1}^L \left( \frac{\beta}{\gamma_k(\beta)} \right)^{m_k} - 1 \right) 
&\xrightarrow{D} N\left(0, \sum_{k=1}^L \left( \sum_{\bm{m} \in M_{m,L}} C_{\bm{m}} m_k \right)^2 \right)
\eal
as $\beta \to \infty$. Combining the above, we get that
\bal
&\sum_{\bm{m} \in M_{m,L}} \sqrt{\beta} \left( \prod_{k=1}^L \left(\frac{\beta}{\gamma_{k}(\beta)}\right)^{m_k} - 1 \right) \frac{ \sum_{i=1}^\infty \prod_{k=1}^L \left( \frac{\gamma_{ki}(\alpha, \beta)}{\beta} \right)^{m_k} }{f(\beta;m,L,c)} \\& \qquad \xrightarrow{D} W \sim N\left(0, \sum_{k=1}^L \left( \sum_{\bm{m} \in M_{m,L}} C_{\bm{m}} m_k\right)^2 \right) 
\eal
as $\alpha, \beta \to \infty$ such that $\frac{\alpha}{\beta} \to c$.  

By Lemma \ref{JointCLTDgroups}, $(Z_k(\alpha, \beta), X_{m,L}(\alpha, \beta)) \xrightarrow{D} N(\bm{0}, \bm{\Sigma^{(k)}})$ for all $1 \leq k\leq L$.
Finally $(X_m^L(\alpha, \beta), Y_m^L(\alpha, \beta)) \xrightarrow{D} N\left(\bm{0}, \begin{pmatrix} \sigma_{2,L}^2 & 0 \\ 0 & \sigma_{1,L}^2 \end{pmatrix} \right)$ and $Y_{m,L}(\alpha, \beta)$ is independent of the vector $(Z_1(\alpha, \beta), \ldots, Z_L(\alpha, \beta))$. 
Therefore for all $m \geq 2$,
\bal
\tilde{H}^L_{m}(\alpha, \beta) \xrightarrow{D} X + Y + W \sim N(0, \sigma_{c,L}^2),
\eal
as $\alpha, \beta \to \infty$ such that $\frac{\alpha}{\beta} \to c$, with variance
\bal
\sigma_{c, D}^2 
&= \sigma_{2,L}^2 + \sigma_{1,L}^2 + \sum_{k=1}^L \left( \sum_{\bm{m} \in M_{m,L}} C_{\bm{m}} m_k \right)^2 \\
&\qquad - 2\sum_{k=1}^L \left(\sum_{\bm{m}_1 \in M_{m,L}} C_{\bm{m}_1} \bm{m}_{1,k} \right)\left(\sum_{\bm{m}_2 \in M_{m,L}} C_{\bm{m}_2} \bm{m}_{2,k} \right) \\
&= \frac{\sum_{j=1}^{2m} \tilde{A}_j(m,D) \frac{\Gamma(j)}{c^{j-1}} - \sum_{1 \leq i,j \leq m} A_i(m,L)A_j(m,L) \frac{\Gamma(i+1)\Gamma(j+1)}{c^{i+j-1}}}{\left( \sum_{j = 1}^{m} A_j(m,L) \frac{\Gamma(j)}{c^{j-1}}\right)^2} \\
&\qquad - \sum_{k=1}^L \left( \sum_{\bm{m} \in M_{m,L}} C_{\bm{m}} m_k \right)^2. \qedhere
\eal
\end{proof}

\section{Final Remarks} \label{FinalRemarks} 

\subsection{} From the gamma representation, we can see that the level one process leads to a random partition of the interval $[0,\beta]$. Replacing the level one process with other random discrete distributions generates different random partitions. It is thus natural to extend our results along these lines. One natural model is the Pitman-Yor process. Furthermore one can even replace the level two process with a Pitman-Yor process, thus giving the hierarchical Pitman-Yor process. Additional developments could involve the more general hierarchical models studied in \cite{CLOP19}.   

\subsection{} Observe that the CLTs in Theorem~\ref{CLTHomozygosityHDP} and Theorem~\ref{CLTHomozygosityHDPgroups} are for an individual $\tilde{H}^L_m(\alpha, \beta)$, with $m \geq 2$ fixed. An immediate extension is a multivariate CLT for 
\bal
\left(\tilde{H}^L_2(\alpha, \beta), \tilde{H}^L_3(\alpha, \beta), \ldots \right).
\eal
This will pave the way for the establishment of a CLT for the hierarchical sampling formula.  

\subsection{} Other possible extensions include multi-level models with more than two hierarchies and hierarchical normalized random measures with independent increments. We intend to pursue these in subsequent research. 

\section*{Acknowledgments}
We wish to thank the two referees for their insightful comments and suggestions which significantly helped to improve the presentation of the paper.


\Address

\end{document}